\documentclass[12pt]{amsart}
\usepackage{}
\usepackage{amssymb}
\usepackage{amsfonts}
 \usepackage{amsfonts,amssymb}

\usepackage{mathrsfs}
\let\mathcal\mathscr

\usepackage[all,ps,cmtip]{xy}

\def\llra{\hbox to 10mm{\rightarrowfill}}

\def\lllra{\hbox to 15mm{\rightarrowfill}}

\def\PA{{\widehat A}}

\def\PE{{\widehat E}}

\def\phi{{\varphi}}

\def\cI{\mathcal{I}}
\def\cD{\mathcal{D}}

\def\cF{\mathcal{F}}
\def\cL{\mathcal{L}}
\def\cO{\mathcal{O}}

\def\cP{\mathcal{P}}

\def\cE{\mathcal{E}}

\def\cM{\mathcal{M}}

\def\cQ{\mathcal{Q}}
\def\cB{\mathcal{B}}
\def\cZ{\mathcal{Z}}
\def\cV{\mathcal{V}}

\let\tilde\widetilde

\DeclareMathOperator{\rank}{rank}

\DeclareMathOperator{\Pic}{Pic}

\DeclareMathOperator{\Supp}{Supp}

\DeclareMathOperator{\Bs}{Bs}

\DeclareMathOperator{\vol}{vol}

\newtheorem{lemm}{Lemma}[section]
\newtheorem{theo}[lemm]{Theorem}

\newtheorem{cor}[lemm]{Corollary}
\newtheorem{prop}[lemm]{Proposition}

\newtheorem*{conj*}{Conjecture}

\theoremstyle{definition}
\newtheorem{defn}[lemm]{Definition}

\newtheorem{rema}[lemm]{Remark}

\newtheorem{setup}[lemm]{}

\newtheorem{qu}[lemm]{Question}

\theoremstyle{remark}
\newtheorem*{remark*}{Remark}
\newtheorem*{note*}{Note}

\def\moins{\mathop{\hbox{\vrule height 3pt depth -2pt
width 5pt}\,}}
\begin{document}
\title{On quint-canonical birationality of irregular threefolds}
\author{Jheng-Jie Chen, Jungkai Alfred Chen, Meng Chen, Zhi Jiang}

\address{Department of Mathematics, National Central University, Chung-Li, Taoyuan, Taiwan}
\email{jhengjie@math.ncu.edu.tw}

\address{Department of Mathematics, National Taiwan University, and Natioanl Center for Theoretical Sciences, Taipei 106, Taiwan}
\email{jkchen@ntu.edu.tw}

\address{School of Mathematical Sciences, Fudan University, Shanghai 200433, P. R. China}
\email{mchen@fudan.edu.cn.}

\address{Shanghai Center for Mathematical Sciences, Jiangwan Campus, Fudan University, Shanghai 200438, P. R. China }
\email{zhijiang@fudan.edu.cn}

\thanks{The first and second authors were partially supported by National Center for Theoretical Sciences and Ministry of Science and Technology. The third author was
supported by National Natural Science Foundation of China (\#11571076, \#11731004) and Program of Shanghai Subject Chief Scientist (\#16XD1400400).  The forth author was supported by National Natural Science Foundation of China (\#11871155, \#11731004) and the program ``Recruitment of global experts''.}

\maketitle
\begin{abstract}Let $X$ be a complex smooth projective threefold of general type. Assume $q(X)>0$. We show that the $m$-canonical map of $X$ is birational for all $m\geq 5$.
\end{abstract}
\section{Introduction}
In the study of birational geometry, the canonical divisor $K$, together with pluri-canonical divisor $mK$ for any $m \in \mathbb{Z}$, plays the central
role. Let $V$ be a nonsingular projective variety of dimension $n$. For any integer $m$, denote by $\varphi_{m,V}$ the $m$-canonical map of $V$.  The
geometry of $\varphi_{m,V}$ and its variants draw a lot of attention in recent years. For varieties of general type, it is  proved, independently by
Hacon--McKernan \cite{H-M}, Takayama \cite{Taka} and Tsuji \cite{Tsuji},  that there is a constant $\tilde{r}_n$ depending only on the dimension $n$
($n>2$) such that $\varphi_{m,V}$ is birational for all $m\geq \tilde{r}_n$.
 The following question is fundamental in birational geometry:

\begin{qu}\label{QA} For any integer $n\geq 3$, find the optimal constant $r_n$ so that, for all nonsingular projective $n$-folds of general type,
$\varphi_{m}$ is birational onto its image for all $m\geq r_n$.
\end{qu}

It is well-known that $r_1=3$ and $r_2=5$.
The results of Jungkai Chen and  Meng Chen show that $r_3 \le 57$ (cf. \cite{CC3, C18}). It is clear that $r_n$ is non-decreasing with respect to $n$.
However it seems to be very difficult to find explicit bounds in higher dimensions.

On the other hand, if one works on irregular varieties of general type, according to the method and the principle developed in \cite{JC-H-mz}, together
with some more recent results of Jiang and Sun (\cite{JS}), there are many interesting results about effective birationality of pluricanonical maps
depending only on the fiber dimensions of the Albanese maps. We summarize some results here:
\begin{itemize}
\item[1)] if $X$ is of general type and of maxiaml Albanese dimension, then $|3K_X|$ induces a birational map (see \cite{CH3} and \cite{JLT});
\item[2)] if $X$ is of general type and a general fiber of the Albanese morphism is of dimension $1$, then $|4K_X|$ induces a birational map (see
    \cite{JS}).
\end{itemize}

Hence it naturally grows out the following question:

\begin{qu}\label{QA} 
Is it true that $\varphi_{m}$ is birational for all $m\geq r_n$ and for all {\it irregular varieties} of general type of dimension $n+1$.
\end{qu}

Question \ref{QA} is known to be true when $n=1$.
The purpose of this article is to prove the following result which gives a positive answer to  the case of $n=2$.

\begin{theo}\label{5K} Let $X$ be a nonsingular irregular threefold of general type. Then $\varphi_{5,X}$ is birational.
\end{theo}

\begin{rema}  This result is optimal since the $4$-canonical map of any 3-fold fibered by surfaces of general type with $K^2=1$ and $p_g=2$ is not birational.
Moreover, the proof of Theorem \ref{5K} is essentially numerical hence it works for all the linear systems $|5K_X\otimes P|$ for all $P\in\Pic^0(X)$.
\end{rema}

Before going into details, we first summarize some more known results about threefolds of general type:
\begin{itemize}
\item[1)] Chen and Hacon proved that if $X$ is an irregular threefold of general type with $\chi(\omega_X)>0$, then $|5K_X|$ induces a biratioal map
    (\cite[Theorem 1.1]{JC-H-mz});
\item[2)] By classical results on pluricanonical maps on surfaces (see Theorem \ref{surface} below), one sees that if $X$ is a threefold of general type
    and the Albanese morphism $a_X$ of $X$ factors through a curve of genus $\geq 2$, then $|5K_X|$ induces a birational map.
\end{itemize}

Combining known results, in order to prove Theorem \ref{5K}, it is sufficient to consider the case when $X$ satisfies the following conditions:

\begin{quote}
\noindent{(\dag)}
{\em
$\chi(\omega_X)\leq 0$ and the image of the Albanese morphism is an elliptic curve $E$. }
\end{quote}

For $X$ satisfying condition $(\dag)$, a general fiber $F$ of $a_X$ is a smooth surface of general type.
We denote by $F_0$ the minimal model of $F$.

We now briefly explain the idea of the proof of Theorem \ref{5K} and the organization of the article.
The method based on GV-sheaves and M-regular sheaves due to Chen-Hacon and Pareschi-Popa is quite powerful to deal with pluricanonical systems of varieties with large irregularity, say varieties of maximal Albanese dimension or varieties of Albanese fiber dimension $1$ (see \cite{CH}, \cite{JC-H-mz}, \cite{JLT}, \cite{JS},\cite{pp1}, \cite{BLNP}). However, the situation here is more complicated and it seems to be difficult to conclude simply by studying the positive properties of the pushforward of pluricanonical sheaves on the Albanese varieties. On the other hand, Reider type theorems work perfectly well for adjoint linear systems on surfaces (see \cite{Laz}).
Here, the idea is to combine both methods. We first show that positivity on the Albanese variety gives strong constraints on the structure of $X$ whose $|5K_X|$ is difficult to study. Then we use twisted bicanonical or tri-canonical system to produce a pencil $|L|$ of $X$ and apply Tankeev's principle, we need to show that $|5K_X|$ induces a birational map on a general member $S\in |L|$ to conclude that $|5K_X|$ is birational. In the last step, we extensively apply a version of Reider theorem due to Langer \cite{Lan} together with volume estimation of big divisors on irregular varieties, which is available due to recent progress of Severi type inequalities (\cite{BPS}, \cite{J}, \cite{jp}).

The structure of the article is as follows.
 In Section 2, we recall some elementary but useful
known facts of surfaces and irregular threefolds.
 We start the study of twisted pluricanonical systems in Section 3. More precisely, we consider the
sheaves of the form ${a_X}_* ( \cO_X(M) \otimes \cI_x)$, where $M$ is an divisor on $X$ such that ${a_X}_* ( \cO_X(M) )$ is ample. In fact, we mainly
consider $M=2K_X$ or $M=3K_X$ in the sequel. Suppose that either ${a_X}_* ( \cO_X(2K) \otimes \cI_x)$ or ${a_X}_* ( \cO_X(3K) \otimes \cI_x)$ is
$M$-regular, then we prove that $|5K_X|$ is birational.

If, on the other hand,  ${a_X}_* ( \cO_X(M) \otimes \cI_x)$ is not $M$-regular, then we prove in Section 4 that
there is a decomposition of $M=L_M+\cD_{M,0}$ into a sum of divisors such that $|M+P|=|L_M|+\cD_{M,P}$ and $\cD_{M,P}-\cD_{M,0} = P$. Moreover, the
pushforward
${a_X}_* ( \cO_X(M) )$ has certain tensor-product structure.
With these type of structures, together with the study of twisted bi-canonical map $|2K_X+P|$,  we are able to prove the main theorem in the case that
$p_g(F) \ne 0$ in Section 5.
Similarly, in Section 6, we are able to prove the theorem by studying twisted tri-canonical map $|3K_X+P|$ when $p_g(F)=0$.

\subsection{Notations}
 Let $X$ be a smooth projective variety. We denote by $a_X: X\rightarrow A_X$ the Albanese variety of $X$ and $\Pic^0(X)$ the Picard variety of $X$. We denote by $\sim$ linear equivalence of divisors and $\equiv$ numerical equivalence of divisors.

\subsection*{Acknowledgements}
The forth author would like to thank  National Center for Theoretical Sciences in Taipei for the  warm hospitality in January 2019.
\section{Preliminary results}
\subsection{Some well-known results on surfaces}
We recall a classical theorem due to Bombieri, Miyaoka and many others (see for instance \cite[Section 7]{BPV}).
\begin{theo}\label{surface} Let $F_0$ be a minimal surface of general type. If $k\geq 5$, then  $\varphi_{k,F_0}$ is a birational morphism onto its
canonical model.
If $k\geq 3$, $\varphi_{k,F_0}$ is birational, except in the following cases:
\begin{itemize}
\item[1)] $(K_{F_0}^2, P_g(F_0))=(2,3)$, then $\varphi_3$ is a morphism of degree $2$ and the canonical model of $F_0$ is a hypersurface of degree $8$
    in $\mathbb{P}(1,1,1,4)$;
\item[2)] $(K_{F_0}^2, P_g(F_0))=(1,2)$, then $\varphi_3$ and $\varphi_4$ are maps of degree $2$ and the canonical model of $F_0$ is a hypersurface of
    degree $10$ in $\mathbb{P}(1,1,2,5)$;
\end{itemize}
\end{theo}

\begin{prop}\label{bps}
\begin{itemize}
\item[1)] If $p_g(F)\neq 0$, then the linear system $|2K_{F_0}|$ is base point free (\cite[Theorem 7.4]{BPV}) .
\item[2)] If $K_{F_0}^2\geq 2$, then $|3K_{F_0}|$ is base point free (\cite[Theorem 5.1]{BPV}).
\end{itemize}
\end{prop}

 In this article, we will frequently apply Reider type results of birationality criterion for adjoint line bundles on surfaces. Let $S$ be a smooth
 projective surface and $D$ a nef  $\mathbb Q$-divisor. Let $x_1$ and $x_2$ be two different closed points of $S$. The following theorem (see \cite{Mas}
 or \cite[Theorem 0.1]{Lan}) gives an effective condition for point separation of  $|K_S+\lceil D\rceil|$.

\begin{theo}\label{adjoint-bundle}
Assume that $D^2>8$. If $|K_S+\lceil D\rceil |$ does not separate $x_1$ and $x_2$, then
there exists a curve $C$ passing through $x_1$ and $x_2$  such that $$(D\cdot C)\leq \frac{4}{1+\sqrt{1-\frac{8}{D^2}}}.$$
\end{theo}

In applying Theorem \ref{adjoint-bundle} in our proof, a crucial step is to estimate the lower bound of volumes of nef $\mathbb Q$-divisors on an irregular threefold.
 The following useful results are special forms of the main result in \cite[Remark 3.3 and Corollary 3.4]{J}.

 \begin{prop}\label{J}
 Let $f: X\rightarrow A$ be a fibration from a smooth projective $n$-fold to an abelian variety.  Assume that $D$ is a nef and big divisor on $X$.
 \begin{enumerate}
 \item  If $D|_{X_s}$ is base point free where $X_s$ is a general fiber of $f$ over its image and that $f_*\mathcal O_X(D)$ is a semi-stable ample vector bundle of rank $r$. Then
 $D^n\geq \frac{n}{r}(D|_{X_s})^{n-1}h^0(X, D)$.
 \item If the connected component of a general fiber of $f$ over its image is a curve of genus $g\geq 2$, then $\vol(D)\geq 2(n-1)!h^0(A, f_*D\otimes Q)$, where $Q\in\Pic^0(A)$ general.
 \end{enumerate}
 \end{prop}

 \begin{lemm}\label{non-hyperelliptic}
 Let $F_0$ be a minimal surface of general type such that the linear system $|mK_{F_0}|$ contains an irreducible and smooth curve  for some $m\geq 2$ and $|(m+1)K_{F_0}|$ induces a birational map. Then a general curve $C\in |mK_{F_0}|$ is not hyperelliptic.
\end{lemm}
\begin{proof}
Let $C\in  |mK_{F_0}|$ be a general member, then $|(m+1)K_{F_0}|_{|C}$ induces a birational map of $C$. On the other hand, considering
$$0\rightarrow K_{F_0}\rightarrow K_{F_0}+C\rightarrow K_C\rightarrow 0,$$ we conclude that $|K_{F_0}+C|_{C}=|(m+1)K_{F_0}|_{C}$ is a sublinear system of
$|K_C|$. Hence $C$ is not hyperelliptic.
\end{proof}
\begin{rema} By \cite[Section 7]{BPV}, we know that if $p_g(F_0)>0$, then $|2K_{F_0}|$ is base point free and $|3K_{F_0}|$ induces a birational map  of $F_0$ except $K_{F_0}^2=1$ and $p_g(F_0)=2$ or $K_{F_0}^2=2$ and $p_g(F_0)=3$. If $p_g(F_0)=0$, then $|4K_{F_0}|$ is always birational and $|3K_{F_0}|$ is always base point free when $K_{F_0}^2\geq 2$. If $p_g(F_0)=0$ and $K_{F_0}^2=1$, then $|3K_{F_0}|$ has no fixed component and a general member is irreducible and smooth (see \cite{Cat}).
\end{rema}
 \subsection{Plurigenera of Irregular Threefolds}
Some cases of the following lemma (see for instance \cite[4.4]{CC1}) are known.

\begin{lemm}\label{plurigenera}
Let $X$ be a smooth irregular threefold of general type. Then we have $P_{k+1}(X)>P_k(X)\geq 1$ for all $k\geq 2$.
\end{lemm}
\begin{proof}
If $X$ is of Albanese fiber dimension $\leq 1$, then by \cite{CH} and \cite{JS}, we know that $V^0(\omega_X, a_X)=\{P\in \Pic^0(X)| h^0(\omega_X\otimes
P)\neq 0\}$ has an irreducible component $T$ of dimension $\geq 1$.
Considering the map $$\bigcup_{P\in T}H^0(X, \omega_X\otimes P)\otimes H^0(X, \omega_X^{k}\otimes P^{-1})\rightarrow H^0(X, \omega_X^{k+1}),$$ we conclude that $P_{k+1}(X)>P_k(X)\geq 1$ for any $k\geq 2$.

We now consider the case when the image of $a_X$ is a curve, where we take the Stein factorization $$a_X: X\xrightarrow{f} C\xrightarrow{g} A_X, $$
and let $F$ be a general fiber of $f$.

First recall that by \cite[Lemma 4.1]{CH1},
$g_*f_*\omega_X^k$ and  $g_*f_*(\omega_X^{k}\otimes Q)$ are non-trivial IT$^0$ sheaves on $A_X$ for all $k \ge 2$ and for all $Q \in \PA_X$.
\footnote{Chen and Hacon proved this when $C$ is an elliptic curve but their argument works in general.}

 If $p_g(F)\neq 0$, then $a_{X*}\omega_X$ is a non-trivial GV-sheaf. Hence there exists $Q\in \PA_X$ such that $H^0(X,
K_X+Q)\neq 0$. A non-zero section
$s\in H^0(X, K_X+Q)$ induces a short exact sequence:
$$0\rightarrow f_*\omega_X^k\xrightarrow{\cdot s} f_*(\omega_X^{k+1}\otimes Q)\rightarrow \cQ_k\rightarrow 0.$$
We notice that $P_{k+1}(F)>P_{k}(F)$ for all $k\geq 2$, hence $\cQ_k$ is a non-trivial IT$^0$ coherent sheaf on $C$.  Hence
\begin{eqnarray*}P_{k+1}(X)&=&h^0(X, (k+1)K_X+Q)=h^0(A_X, g_*f_*(\omega_X^{k+1}\otimes Q))\\
&=& h^0(A_X, g_*f_*\omega_X^{k})+h^0(A_X, \cQ_{k})\\
&>& h^0(A_X, g_*f_*\omega_X^{k})=P_k(X).
\end{eqnarray*}

If $p_g(F)=0$, then $q(F)=0$ and hence $f_*\omega_X=R^1f_{*}\omega_X=0$. We have
\begin{eqnarray*}\chi(X, \omega_X)&=&\deg(f_*\omega_X)-\deg(R^1f_*\omega_X)+
\deg(R^2f_*\omega_X)\\&=&\deg\omega_C\geq 0.
\end{eqnarray*}
We denote by $Y$ the minimal model of $X$. Then $\chi(Y, \cO_Y)=\chi(X, \cO_X)\leq 0$. Since  $K_Y^3>0$, by Reid's Riemann-Roch formula, $P_2(X)=\frac{1}{2}K_Y^3-3 \chi(Y, \cO_Y) +\mathcal{R}_2 >0$ and for all $k \ge 2$,
$$P_{k+1}(X)-P_k(X)=\frac{k^2}{2}\cdot K_Y^3-2\chi(Y, \cO_Y)+\mathcal{R}_{k+1}, $$
where $\mathcal{R}_2, \mathcal{R}_{k+1}$ denotes the contribution of singularities which are always non-negative.
It follows immediately that $P_{k+1}(X)> P_k(X) \ge 1 $ for all $k \ge 2$.
\end{proof}

\section{Twisted Pluricanonical systems}
\begin{prop}\label{bir1}Let $X$ be a smooth projective threefold of general type satisfying Condition $(\dag)$. If there exists an open dense subset $U$
of $X$
such that, for each point $x\in U$, $E_{3,x}=a_*(\omega_X^3\otimes\cI_x)$ is M-regular,  then $\phi_{5,X}$ is birational.
\end{prop}
\begin{proof}
For any $x\in U$, $E_{3, x}$ is M-regular and hence is continuously globally generated
(see \cite[Proposition 2.13]{pp1}). Then, by \cite[Corollary 2.4]{JC-H-mz}, there exists an open dense subset $V$ of $U$ such that the linear system
$|5K_X|$ separates any two points $x$, $y$ of $V$ in different fibers of $a_X$. Moreover, if $x, y\in V$ are two different points in a general fiber $F$
of $a_X$ and if $|3K_F|$ separates $x$ and $y$, then $|5K_X|$ also separates $x$ and $y$.

By Theorem \ref{surface}, $\varphi_{3,F}$ is birational unless $F$ is either a $(1,2)$-surface or a $(2,3)$-surface. Moreover, for these two classes of
surfaces, $|3K_F|$ induces a map of degree $2$.

Assume, to the contrary, that $\varphi_{5,X}$ is non-birational. We want to deduce a contradiction.

Then the linear system $|5K_X|$ induces a map $\phi_5: X\rightarrow Y$ of degree $2$. We are going to argue by contradiction to show that it is absurd.

Note that $\varphi_{5,X}:X\rightarrow \hat{Y}$ must be generically finite of degree $2$. Hence there is
a birational involution $\tau$ on $X$ switching two fibers of $\phi_5$.
Modulo birational modifications, we may assume that $\tau$ is a biregular involution and the quotient ${Y}=X/\tau$ is also a smooth projective variety.
The involution $\tau$ preserve general fibers of the Albanese morphism $a_X$ and hence we have a commutative diagram:
\begin{eqnarray*}
\xymatrix{
X\ar[r]^{f}\ar[dr]_{a_X}& {Y}\ar[d]^{a_Y}\\
& E,}
\end{eqnarray*}
where $f$ is the quotient morphism of $\tau$.  By assumption, $\varphi_{5,X}$ factors through $f$.
There exists a line bundle $\cL$ on $Y$ such that $\omega_X=f^*(\omega_Y\otimes\cL)$ and $f_*\cO_X=\cO_Y\oplus \cL^{-1}$. Thus we have
$$f_*\omega_X^5=\big(\omega_Y^{\otimes 5}\otimes\cL^{\otimes 5}\big) \oplus\big( \omega_Y^{\otimes 5}\otimes\cL^{\otimes 4}\big),$$ where
$\omega_Y^{\otimes 5}\otimes\cL^{\otimes 5}$ (resp. $\omega_Y^{\otimes 5}\otimes\cL^{\otimes 4}$) correspond to $\tau$-invariant (resp.
$\tau$-anti-invariant) sections of $\omega_X^5$.

Since $|5K_F|$ induces a birational map of $F$ and ${a_X}_*\omega_X^{\otimes 5}$ is ample,  both  ${a_Y}_*\big(\omega_Y^{\otimes 5}\otimes\cL^{\otimes
5}\big)$ and
${a_Y}_*\big( \omega_Y^{\otimes 5}\otimes\cL^{\otimes 4}\big)$ are ample. Since both the ample sheaves on $E$ have global sections, they are lift to
global ones of $\omega_X^{\otimes 5}$. Hence
direct summand of $a_{X*}\omega_X^{\oplus 5}$, it is ample. Hence
$H^0(X, 5K_X)$ contains a section separating the two points on the general fiber of $f$,  which means however that $\varphi_{5,X}$ does not factor
birationally through $f$, a contradiction.
\end{proof}

Hence, in order to prove the main theorem,  we will mainly work under the condition that $E_{3, x}$ is not M-regular in the rest of the article. But some
variant of Proposition \ref{bir1} is also needed to deal with some special cases.

\begin{prop}\label{2K-not-m-regular} Let $X$ be a smooth projective threefold of general type satisfying Condition $(\dag)$. Let $F$ be a general fiber of
$a_X$. Assume that, for  general point $x\in X$, $a_{X*}(\omega_X^2\otimes\cI_x)$ is M-regular.  Then $|5K_X|$ induces a birational map.
\end{prop}
 \begin{proof}
  Since $a_{X*}(\omega_X^2\otimes\cI_x)$ is M-regular for general $x\in X$, the evaluation map
$$\bigoplus_{P\in W}H^0(X, a_{X*}(\omega_X^2\otimes\cI_x)\otimes P)\otimes P^{-1}\rightarrow a_{X*}(\omega_X^2\otimes\cI_x)$$ is surjective for any open
subset $W$ of $\Pic^0(X)$.
 Since $p_2(F)\neq 0$ for general fiber $F$ of $a_X$ and $\Bs|2K_F|$  is a proper closed subset of $F$, for any two general points $x$, $y$ of $X$ lying
 in different general fibers of $a_X$, there exists a point $P\in W$ and a divisor  $D_P\in |2K_X+P|$ passing through $x$ and not passing through $y$. As
 $y\in X$ is general, there exists a smaller open subset $W'\subset W$ such that $y$ is not contained in the base locus of $|3K_X-Q|$ for any $Q\in W'$.
 Thus, by considering the natural map $|2K_X+P|\times |3K_X-P|\rightarrow |5K_X|$, there exists a section in $H^0(5K_X)$ separating $x$ and $y$.

{}From now on, we always assume that $x$ and $y$ are two general points in a general fiber $F$ of $a_X$.

If $|2K_F|$ can separate $x$ and $y$, then $|5K_X|$ can separate $x$ and $y$ by the argument of \cite[Corollary 2.4]{JC-H-mz}.
 Assume that $|3K_F|$ induces a birational map. 
 Take $x, y \in F$
 to be two general points in a general fiber $F$ and that $x$ and $y$ can not be separated by $|2K_F|$.
Since $a_{X*}\omega_X^3$ is IT$^0$,  it is M-regular and, by taking the localization,  the natural map $$\bigoplus_{P\in W}H^0(X, 3K_X+P)\rightarrow
H^0(F, 3K_F)$$ is surjective for any open subset $W$ of $\Pic^0(X)$.
Find a section $f\in H^0(F, 3K_F)$ such that $f(x)=0$ and $f(y)\neq 0$. 
We then choose $\tilde f_i\in H^0(X, 3K_X+P_i)$, $1\leq i\leq k$,  such that $\sum_i\tilde f_i\mid_F=f$.
Since $a_{X*}(\omega_X^2\otimes \cI_x)$ is M-regular, $V^1({a_X}_*(\omega_X^2\otimes \cI_x))=\emptyset $.  For a general $x\in X$, the short exact sequence
$$0\rightarrow {a_X}_*(\omega_X^2\otimes \cI_x)\otimes P \rightarrow {a_X}_*\omega_X^2\otimes P\rightarrow \mathbb C_{x}\rightarrow 0$$ implies the
surjective map
$H^0(2K_X+P)\rightarrow \mathbb C_{x}$ for any $P\in \Pic^0(X)$, which means
that, as a general point of $X$, $x$ is not in the base locus of $|2K_X+P|$ for any $P\in \Pic^0(X)$.
We then take $\tilde g_i\in H^0(X, 2K_X-P_i)$, $1\leq i\leq k$, such that $\tilde g_i(y)=1$ for $1\leq i\leq k$. As $x$ and $y$ can not be separated by
$|2K_F|$, we see that $\tilde g_i(x)\neq 0$  for  $1\leq i\leq k$. We then conclude that $h:=\sum_{1\leq i\leq k}\tilde g_i\tilde f_i\in H^0(X, 5K_X)$
satisfies the property that $h(x)= 0$ and $h(y)\neq 0$.

When $|3K_F|$ does not induce a birational map for $F$, the argument of the second part in the proof of Proposition \ref{bir1} works as well. So we have
completed the proof.
 \end{proof}

\section{Abel-Jacobi map induced by Fixed divisors}
In this section, we use Abel-Jacobi map to analyze the fixed divisors of twisted pluricanonical systems of an irregular variety $X$ with $q(X)=1$. Lemma
\ref{bundle} and \ref{bundle1} will be frequently used in next sections.

\begin{setup}
{\bf Setting.}
Let $X$ be a smooth projective variety of general type and assume that the Albanese morphism $a_X: X\rightarrow E$ of $X$ is a fibration onto an elliptic
curve. Denote by $F$ a general fiber of $a_X$.

Denote by $\cP$ the Poincar\'{e} bundle, on $X\times \PE$, which is pulled back from $E\times\PE$.
Assume that we have a divisor $M$ on $X$ such that  $H^0(X, M\otimes P)\neq 0$ for general $P\in\PE$.
Then $p_{2*}(p_1^*M\otimes\cP)$ is a non-trivial vector bundle on $E$, where $p_1$ and $p_2$ are  the natural projections from $X\times \PE$
to $X$ and $\PE$ respectively.
We have the relative evaluation map:
\begin{eqnarray}\label{evaluation}p_2^*p_{2*}(p_1^*M\otimes\cP)\xrightarrow{v} p_1^*M\otimes\cP.
\end{eqnarray}
Then there exists an ideal sheaf $\cI_M$ on $X\times \PE$ so that $\cI_M\otimes p_1^*M\otimes\cP$ is the image of $v$. Let $\cZ_M$
be the scheme of $ X\times \PE$  defined by $\cI_M$.
\end{setup}

The following lemma is a crucial observation.  Note that a similar observation plays a crucial role in \cite{BLNP} as well, though the context is completely different.

\begin{lemm}\label{regularity} Keep the setting as in 4.1.
Assume that $a_{X*}\cO_X(M)$ is an ample vector bundle on $E$ and $\cZ_M$ does not dominate $X$. Then there exists an open dense subset $U$ of $X$ such
that
for any $x\in U$, $a_{X*}(\cO_X(M)\otimes\cI_x)$ is M-regular. In particular, $a_{X*}(\cO_X(M)\otimes\cI_x)$ is ample on $E$.
\end{lemm}

\begin{proof}
For  a general $P\in\PE$,  one has $$p_{2*}\cO_X(p_1^*M\otimes\cP)\otimes \underline{\mathbb C}_{P}\simeq H^0(X, M\otimes P).$$
Thus, if $\cZ_M$ does not dominate $X$, the following base locus
$$\cB:=\{(x, P)\in X\times \PE\mid x\in \Bs|M+P|\}$$ does not dominate $X$ and denote by  $U$ the open dense subset $X\moins p_1(\cB)$.
Then, for any $x\in U$, we have the short exact sequence
$$0\rightarrow a_{X*}(\cO_X(M)\otimes\cI_x)\rightarrow a_{X*}\cO_X(M)\rightarrow  \underline{\mathbb C}_{ a_X(x)}\rightarrow 0.$$
By definition of $\cB$, we conclude that $H^0(X, a_{X*}\cO_X(M)\otimes P)\rightarrow \mathbb C_{ a_X(x)}$
is surjective for any $P\in\PE$. By assumption, $a_{X*}\cO_X(M)$ is an ample vector bundle on $E$.
By considering the long exact sequence, we conclude that $H^1(E, a_{X*}(\cO_X(M)\otimes\cI_x)\otimes P)=0$ for any $P$.
Thus $a_{X*}(\cO_X(M)\otimes\cI_x)$ is M-regular and, in particular, ample on $E$.
\end{proof}

Now the rest of this section is devoted to studying in detail the case when $\cZ_M$ dominates $X$ via the first projection. We denote by $\cD_M$ the union
of irreducible components, of $\cZ_M$, each of which dominates $X$. We observe that each irreducible component of $\cD_M$ is a divisor of $X\times\PE$
dominating both $X$ and $E$.

\begin{lemm}\label{integral} Keep the setting as in 4.1. Assume that  $\cZ_M$ dominates $X$. Then $\cD_M$ is reduced and irreducible and, for  $P\in\PE$
general , $\cD_{M,P} \colon= \cD_M|_{X \times\{P\}}$ is also reduced and irreducible.
Moreover, the Abel-Jacobi map induced by $\cD$
\begin{eqnarray*}
&\Psi:&\PE\longrightarrow\PE\\
&&P\rightarrow \cO_X(\cD_{M,P}-\cD_{M,0})
\end{eqnarray*}
is the identity and hence $\cD_{M,P}=\cD_{M,0}+P$. Fix a general $P$, we set $L_M \colon= M+P-\cD_{M,P}$. Then $|M+Q|=|L_M|+ \cD_{M,Q}$ for general $Q$.
Furthermore, $V^0(L_M) \ne \PE$.
\end{lemm}
\begin{proof}
First we prove that each irreducible component of $\cD_M$ is reduced.
Otherwise let $k\cD_1$ (with $k\geq 2$) be an irreducible component of $\cD_M$.
Fix a general element $P\in \PE$, then $k\cD_{1P}$ is a fixed divisor on $X$.
We notice that $\cD_{1}$ is a family of divisors of $X$ parametrized by $\PE$
and the Abel-Jacobi maps:
\begin{eqnarray*}
&\Phi:&\PE\longrightarrow\PE\\
&&P\rightarrow \cO_X(\cD_{1P}-\cD_{10})
\end{eqnarray*}
is a non-trivial isogeny of $\PE$. Hence we may take two general elements $Q_1$ and $Q_2$, different from  $P$, such that $\Phi(Q_1)+\Phi(Q_2)=2\Phi(P)$.
Hence $\cD_{1Q_1}+\cD_{1Q_2}$ is rational equivalent to $2\cD_{1P}$. This contradicts the fact that $k\cD_{1P}$ is a fixed divisor.

If $\cD_M$ has two different irreducible components $\cD_1$ and $\cD_2$. We still consider the Abel-Jacobi maps
\begin{eqnarray*}
&\Phi_i:&\PE\longrightarrow\PE\\
&&P\rightarrow \cO_X(\cD_{iP}-\cD_{i0}),
\end{eqnarray*}
for $i=1,2$. Both $\Phi_1$ and $\Phi_2$ are isogenys of $\PE$. Fix a general $P\in\PE$, then $\cD_{1P}+\cD_{2P}$ is a fixed divisor. Choose a general
$Q\in\PE$ such that for some
$P_1\in \Phi_1^{-1}(\Phi_1(P)-Q)$ and $P_2\in \Phi_2^{-1}(\Phi_2(P)+Q)$, the support of  $\cD_{1P_1}+\cD_{2P_2}$ is
different from the support of $\cD_{1P}+\cD_{2P}$. However $\cD_{1P}+\cD_{2P}\sim_{\textrm{rat}} \cD_{1P_1}+\cD_{2P_2}$, which is a contradiction as well.

Take the normalization $\varepsilon: \cD'\rightarrow \cD_M$. Let $\cD'\xrightarrow{g} C\xrightarrow{t} \PE$ be the Stein factorization of
$\cD'\rightarrow \PE$. If $\cD_{M,P}$ is not irreducible for the general $P\in \PE$, then $\deg t>1$.
Then we consider another Abel-Jacobi map:
\begin{eqnarray*}
&\Phi_C:&C\longrightarrow\PE\\
&&P\rightarrow \cO_X(\varepsilon(\cD'_{c})-\varepsilon(\cD'_{c_0})),
\end{eqnarray*}
for some fixed point $c_0\in C$. Clearly $\Phi_C$ is dominant as well.

For the general $P\in \PE$,  take two different points $c_1, c_2\in t^{-1}(P)$. We have $\varepsilon(\cD'_{c_1})+\varepsilon(\cD'_{c_2})$ is a sub-divisor
of $\cD_P$ and hence a fixed divisor. However, since  $\Phi_C$ is dominant,
there are different points  $c_3, c_4\in C$ such that  $\Phi_C(c_1)+\Phi_C(c_2)=\Phi_C(c_3)+\Phi_C(c_4)$ and then
$\varepsilon(\cD'_{c_1})+\varepsilon(\cD'_{c_2})\sim_{\textrm{rat}} \varepsilon(\cD'_{c_3})+\varepsilon(\cD'_{c_4}),$
a contradiction. Hence $\cD_{M,P}$ is irreducible and reduced for the general $P\in \PE$.

The above argument actually shows that $\Psi$ is injective and hence a group isomorphism from $\PE$ to $\PE$. We prove the last statement by the same
idea. Assume that $\Psi$ is different from the identity, then $\Psi$ is induced by a complex multiplication of $\PE$ and $\Psi-Id$ is also a non-trivial
group isomorphism of $\PE$.
We take a general $P_0\in \PE$ such that $h^0(X, M\otimes P_0)$ is minimal. We denote by $L_M$ the divisor $M+P-\cD_{M,P_0}$. Then, for a general $Q\in
\PE$, there exists $Q'\in \PE$ different from $Q$, such that $\Psi(Q')=\Psi(P_0)+Q-P_0$. This implies that $\cD_{Q'}$ linear equivalent to
$\cD_{P_0}+Q-P_0$. Then $|M+Q|\supset |L_M|+\cD_{Q'}$. Since $Q\in \PE$ general, we conclude that $|M+Q|=|L_M|+\cD_{Q'}$. On the other hand, we know that
$\cD_Q\neq \cD_{Q'}$ is a fixed divisor of the linear system $|M+Q|$, thus $L_M-D_{Q}$ is effective, which is absurd.
 \end{proof}

We continue the study under the assumption of Lemma \ref{integral}. 
Notice that, whenever we take $M=kK_X+a_X^*{\hat{L}}$ with $k\geq 2$, $P_k(F)\neq 0$ and $\hat{L}$ a nef line bundle on $E$, $a_{X*}\cO_X(M)$ is ample
(see for instance \cite[Lemma 2.1]{CH1}). Then $h^0(X, M+P))=r$  is constant for all $P\in\PE$ and hence $|M+P|=|L_M|+\cD_{M,P}$ for all $P\in\PE$.

We now consider the following vector bundles on $E$:
$$\cV_M^0:=a_{X*}\cO_X(L_M) \text{ and } \cV_M^+=a_{X*}\cO_X(\cD_{M,0}).$$
We know that a vector bundle on an elliptic curve is a direct sum of semi-stable vector bundles.
Since $V^0(L_M)\neq \PE$ and $h^0(X, L_M)=r$, $\cV_M^0$ is a direct sum of semi-stable vector bundles of negative slopes and zero slope.
In particular, there is an injection
 $$\cE_M^0:=\cO_E^{\oplus r}\hookrightarrow \cV_M^0$$ which is induced by global sections of $L_M$.

Similarly, $\cV_M^+$ is a direct sum of semistable vector bundle. Since
$$h^0(E, \cV_M^+ \otimes P)=h^0(X, \cD_{M,0} \otimes P)=h^0(X, \cD_{M,P})=1$$ for any general $P\in\PE$, there exist one (and only one) ample direct
summand $\cE_M^{+}$
of $\cV_M^+$  of degree $1$.

We have the natural map $a_*\cO_X(L_M)\otimes a_*\cO_X(\cD_{M,0})\rightarrow a_{X*}\cO_X(M)$ induced by multiplication of sections. Hence we have the
multiplication map
\begin{eqnarray}\label{multiplication}m: (\cO_E^{\oplus r})\otimes \cE_M^+   \hookrightarrow \cV_M^0 \otimes \cV_M^+ \rightarrow    a_{X*}\cO_X(M).
\end{eqnarray}

Recall that $|M+P|=|L_M|+\cD_{M,P}$. Therefore $m$ induces an injective and hence bijective map between sections
$$H^0(E, (\cO_E^{\oplus r})\otimes \cE_M^{+}\otimes P)\xrightarrow{m} H^0(E, a_{X*}\cO_X(M)\otimes P),$$ for any $P\in\PE$. Notice also that both vector
bundles are ample and hence have vanishing $H^1$.
Let $\cF^\cdot \in {\bf D}(E)$ be the mapping cone of the map $m$ in (\ref{multiplication}).  It follows that $H^i(E, \cF^\cdot \otimes P) =0$ for all $i$ and all
$P$. Hence $\cF^\cdot=0$ and therefore $m$ is an isomorphism between vector bundles on $E$.

In summary, we have proved the following lemma.

\begin{lemm}\label{bundle} Assume that $M$ is a divisor on $X$ such that $a_{X*}\cO_X(M)$ is an ample vector bundle on $E$ of degree $r$ and the scheme
$\cZ_M$ defined in (\ref{evaluation}) dominates $X$.
Then we can write $|M|=|L_M|+\cD_{M,0}$ such that $|M+P|=|L_M|+\cD_{M,P}$ for all $P$,
${a_X}_*\cO_X(\cD_0)$ contains an ample vector bundle of degree $1$, denoted $\cE_M^+$,  and
${a_X}_*\cO_X(L_M)$ contains $ \cE_M^0=\cO_E^{\oplus h^0(M)}$.

Furthermore, we have an isomorphism of vector bundles
$$m: \cE_M^0 \otimes \cE_M^+ =(\cO_E^{\oplus h^0(X, M)})\otimes \cE_M^+\xrightarrow{\sim} a_{X*}\cO_X(M).$$
This implies, in particular, that
$$  h^0(X, M) \cdot \textrm{rk}( \cE_M^+) = h^0(F, M|_F).$$
\end{lemm}

\begin{defn}
We call the decomposition $|M|=|L_M|+\cD_{M,0}$ in Lemma \ref{bundle}  {\it a uniform decomposition of $|M|$}, in the sense that the decomposition $|M+P|=|L_M|+\cD_{M,P}$ holds for all $P$ uniformly. Moreover, $a_*\cO_X(M) \cong \cE_M^0 \otimes \cE_M^+$ is called  {\it a uniform decomposition of $a_* \cO_X(M)$}.
\end{defn}

 It is possible to have further decomposition. Let $B$ be a fixed divisor of $|L_M|$ and $L'_M=L_M-B$. That is, $|L_M|=|L'_M|+B$. Since one has $V^0(L'_M) \subset V^0(L_M) \ne \PE$, and $h^0(L'_M)=h^0(L_M)$, it follows that  $\cE_M^0 \subset a_* \cO_X(L'_M) \subset a_* \cO_X(L_M)$. Moreover,
$$ \cE_M^0 \otimes \cE_M^+ \stackrel{\cdot B}{\longrightarrow} a_* \cO_X(M)$$ is an isomorphism.

\begin{lemm} \label{modification} Given a divisor $M$ on $X$ such that $a_* \cO_X(M)$ is ample and $\cZ_M$ dominates $X$. After replacing $X$ by its birational model
and replacing $M$ by its total transform, there exists a uniform decomposition $|M| = |L_M| + \cD_{M,0} + E_M$ such that ${L}_M$ is base point free and ${\cD}_{M,0}$ is nef and smooth. Moreover, there is an isomorphism of vector bundles
$$ \cE_M^0 \otimes \cE_M^+ \stackrel{\cdot E_M}{\longrightarrow} a_* \cO_X(M).$$

This decomposition  is called {\it a refined uniform decomposition}. The induced isomorphism of vector bundles  is called {\it a refined uniform decomposition} of $a_*\cO_X(M)$.
\end{lemm}

\begin{proof}
Consider $\pi: \tilde{X} \to X$  a birational morphism resolving the base loci of $|L_M|$ and resolving $\cap_{P \in \PE} \cD_{M,P}$. Let $\tilde{L}_M$ and $\tilde{\cD}_{M,P}$ be the proper transform of $L_M$ and $\cD_{M,P}$ respectively.
Note that for any curve $C$, if $C \subset \text{Supp}(\tilde{\cD}_{M,P})$, then $C \not \subset \text{Supp}(\tilde{\cD}_{M,Q})$ for general $Q$.
Hence $C \cdot (\tilde{\cD}_{M,P}) = C \cdot (\tilde{\cD}_{M,Q})  \ge 0$, that is $\tilde{\cD}_{M,P}$ is nef.

Similarly, one can show that there exists $\tilde{\cE}_M^0$, $\tilde{\cE}_M^+$ subbundles of $a_* \pi_* \cO_{\tilde{X}}(\tilde{L}_M)$ and $a_* \pi_* \cO_{\tilde{X}}(\tilde{\cD}_{M,0})$ respectively such that
$$ \tilde{\cE}_M^0 \otimes \tilde{\cE}_M^+ \hookrightarrow a_* \cO_{\tilde{X}}(\tilde{L}_M) \otimes a_* \cO_{\tilde{X}}(\tilde{\cD}_{M,0}) \stackrel{\cdot E_M}{\longrightarrow} a_*\cO_{\tilde{X}} (\pi^*M) = a_*\cO_X(M)$$ is an isomorphism.
Note that $h^0(\tilde{X}, \tilde{L}_M)= h^0(X, L_M)= h^0(X, M)$, hence $\tilde{\cE}_M^0 ={\cE}_M^0$ and therefore $\tilde{\cE}_M^+ \cong {\cE}_M^+$.
 \end{proof}


We will need the following facts.
\begin{lemm} \label{decomp}
Let $L$ be a line bundle on a smooth projective variety. Assume that $L=L_1\otimes L_2$ is the tensor product of two line bundles and there exists
subspaces $W_1\subset H^0(L_1)$ and $W_2\subset H^0(L_2)$ such that the natural map $W_1\otimes W_2\rightarrow H^0(L)$ is an isomorphism. Then the following holds.
\begin{enumerate}
\item If $L$ is free, then both $L_1$ and $L_2$ are free.
\item If $\dim W_1=1$,  then $W_1=H^0(L_1)$, $W_2=H^0(L_2)$ and $|L|=|L_2|+D_1$, where $D_1$ is the unique effective divisor in $|L_1|$.
\end{enumerate}
\end{lemm}
\begin{proof}
(1). Let $s_1,\ldots s_n$ and $t_1, \ldots t_m$ be basis of $W_1,W_2$ respectively. Then $\{ s_i t_j\}$ forms a basis of $H^0(L)$. Suppose that $s_1,\ldots s_n$ has a common zero at $x \in X$, then  $x$ is a common zero of $ \{ s_i t_j\}$ for all $i,j$, which is absurd.
Hence $W_1$ is free and so is $L_1$.
The same argument holds for $W_2$ and $L_2$.

(2). Since $\dim W_1=1$, we then have $h^0(L)=\dim W_2$. On the other hand, $h^0(L)\geq h^0(L_2)\geq \dim W_2$. Hence $W_2=H^0(L_2)$.
Let $0\neq s \in W_1$ and let $D$ be the divisor $\text{div}(s)$. Then $W_1\otimes W_2=H^0(L)$ implies that $D$ is a fixed divisor of the linear system $|L|$. In
particular, $h^0(L_1)=h^0(D)=1$.
\end{proof}
\begin{cor}\label{product} Apply the above Lemma to a refined uniform decomposition $|M|=L_M+ \cD_{M,0}+E_M$,  one has the following:
\begin{enumerate}

\item If $\textrm{rk}(\cE_M^+)=1$ or $\textrm{rk} (\cE_M^0)=1$, then $\cE_M^+=\cV_M^+$ and $\cE_M^0=\cV_M^0$.


\item If $\textrm{rk} (\cV_M^0) =1$, then  $\textrm{rk} (\cE^0)=1$ and hence $h^0(X,M)=1$.


\item Let $L_{M,F}$ and $\cD_{M,F}$ denote the restriction of $L_M$ and
    $\cD_{M,0}$ to general fiber $F$ respectively. If $|M_F|$ is free, then both $|L_{M,F}|$ and $|\cD_{M,F}|$ are free,

\item  If $\textrm{rk}(\cE_M^+) =1$, then $\cD_{M,0} \equiv F$ for a general fiber $F$. Moreover, the map restricting to general fiber $H^0(X, M) \to
    H^0(F, M_F)$ is surjective.

 \end{enumerate}
\end{cor}

\begin{proof} The first three statements are clear. We  prove the last statement.
If  $\cE_M^+$ is a line bundle of degree $1$, denoted $\cO_E(p)$ for some $p \in E$. Then  $ \cO_E(p)
\subset a_* \cO_X(\cD_{M,0})$ and hence $a_X^* \cO_E( p) \otimes P \subset \cD_{M,P}$ for all $P$. For general $P$, $\cD_{M,P}$ is irreducible containing the fiber $F$ such that $\cO_X(F)= a_X^* \cO_E( p) \otimes P$. In particular, $\cD_{M,0} \equiv F$.

Now $a_{X*}\cO_X(M)=(\cO_E^{\oplus h^0(X, M)})\otimes \cE_M^+ = {\cE_M^+}^{\oplus h^0(X, M)}$, which is clearly generically generated by global sections. In particular,
the map restricting to general fiber $H^0(X, M) \to H^0(F, M_F)$ is surjective.
\end{proof}

\begin{lemm}\label{bundle1} Keep the same  assumption as that of Lemma \ref{bundle}.
If $M=M_1+M_2$ such that $a_{X*}\cO_X(M_1)$ is a non-trivial
nef vector bundle and $a_{X*}\cO_X(M_2)$ is an ample vector bundle. Let $M=L_M+\cD_{M,0}$ be a uniform decomposition. Then the following holds.
\begin{enumerate}
\item  There exists an element $Q_0 \in\Pic^0(E)$
such that $V^0(M_1)=\{Q_0\}$;

\item  $M_2-Q_0$ admits a uniform decomposition $M_2-Q_0=L_{M_2} + \cD_{M,0}$ such that $|M_2-Q_0+P|=|L_{M_2}|+\cD_{M,P}$ for all $P\in\PE$;

\item  $a_{X*}\cO_X(M_2-Q_0)= \cE_{M_2}^0 \otimes \cE_{M}^+$, where $\cE_{M_2}^0=\cO_E^{\oplus h^0(X, M_2)}$;

\item $ h^0(X, M_2) \cdot \textrm{rk} (\cE_{M}^+) = h^0(F, M_{2,F}),$
where $M_{2,F}$ is the restriction of $M_2$ to the general fiber $F$.

\item $L_M=L_{M_2}+Q_0+M_1$ and the restriction to $F$ yields $L_{M,F}=L_{M_2,F}+M_{1,F}$.
\end{enumerate}
\end{lemm}
\begin{proof}
Since $a_{X*}\cO_X(M_1)$ is a non-trivial nef vector bundle, there exists $Q_0 \in\PE$ such that $H^0(X, \cO_X(M_1)\otimes Q_0)\neq 0$ and we fix a
divisor $G_{Q_0} \in |M_1+Q_0|$. We have $G_{Q_0}+|M_2-Q_0+P|\subset |M+P|=|L_M|+\cD_{M,P}$.
Since $\cD_{M,P}$ is irreducible and $\cD_M$ dominates $X$, therefore for general $P$, $\cD_{M,P} \nsubseteq \Supp(G_{Q_0})$ and hence $\cD_{M,P}$ is a
fixed part of the linear system $|M_2-Q_0+P|$. Since $h^0(X, \cO_X(M_2)\otimes P)$ is constant for every $P\in \PE$, $\cD_P$ is a fixed part for all
$|M_2-Q_0+P|$. By the similar argument as in Lemma \ref{bundle}, we see the desired structure of
$a_{X*}\cO_X(M_2-Q_0)$.

If there exists $Q_1\in V^0(M_1)$ different from $Q_0$, then by the same argument we conclude that $\cD_{M, P}$ is a fixed divisor of $|M_2-Q_1+P|$
for $P\in \PE$. We then conclude that the Abel-Jacobi map for the fixed divisors of $|M_2+P|$ is not the identity of $\PE$, which contradicts  Lemma \ref{integral}. Hence $V^0(M_1)=\{Q_0\}$.
\end{proof}

In the rest of this article, we focus on the following threefolds.

\begin{defn}
A threefold $X$ is said to be a {\it special irregular threefold}, if $X$ is a threefold of general type  satisfying Condition $(\dag)$, i,e, $\chi(X,
\omega_X) \le 0$ and the Albanese image is an elliptic curve,  and for $x\in X$ general, both $E_{2,x}:=a_{X*}(\omega_X^2\otimes \cI_x)$ and
$E_{3,x}:=a_{X*}(\omega_X^3\otimes \cI_x)$ are not M-regular.
\end{defn}

We shall work on special irregular threefolds from now on.

According to Lemma \ref{regularity}, $\cZ_{3K_X}$ dominates $X$ when $X$ is a special irregular threefold.  We now apply
Lemma \ref{bundle} and get a uniform decomposition $3K_X=L_{3K}+\cD_{3K,0}$ with $|3K_X|=|L_{3K}|+\cD_{3K,0}$ and
$$a_{X*}\cO_X(3K_X)=\cE_{3K}^0 \otimes \cE_{3K}^+. \eqno{\ddagger_3}$$
For the general fiber $F$ of $a_X$, we may write the restriction as $3K_F=L_{3K,F}+\cD_{3K, F}$. We will denote by $F_0$ the minimal model of $F$ and
denote by $\sigma: F\rightarrow F_0$ the contraction morphism.

We have $$H^0(F, 3K_F)\simeq W_3^0 \otimes W_3^+ ,$$ where $W_3^0$ (resp. $W_3^+$) is the corresponding subspace of $ H^0(F, L_{3K,F})$ (resp. $H^0(F, \cD_{3,F})$) and
since $a_{X*}\cO_X(L_{3K})$ contains $\cE_{3K}^0= \cO_E^{h^0(X,3K_X)}$ we have
$$h^0(F, L_{3K,F}) = \textrm{rk}(a_{X*}\cO_X(L_3)) \ge  \textrm{rk}(\cE_{3K}^0) = P_3(X)>1.$$


Let $\varphi \colon X \to Z$ be the Stein factorization of the bi-canonical (resp. tri-canonical)  map if $P_2(X) \ge 2$ (resp. $P_3(X) \ge 2$). Let $S$ be the general hyperplane section if $ \dim \varphi(X) \ge 2$ or general irreducible fiber of $\varphi$ if $\dim \varphi(X)=1$. Then we
will prove that $|5K+P|$ is birational by studying the restriction $|5K|_S$ in the remaining part of the article.\\

\section{Irregular threefolds whose Albanese fibers have $p_g(F)\neq 0$}

The aim of this section is to prove the main theorem for special irregular threefolds with $p_g(F) \ne 0$.
If $p_g(F)\neq 0$, then $a_*\cO_X(K_X)$ is non-trivial.

Considering the divisor $3K_X=K_X+2K_X$ and refined uniform decomposition $|3K_X+P|=|L_{3K}|+\cD_{3K, P}+E_3$, we may apply Lemma \ref{bundle} and Lemma \ref{bundle1} to get:
$$\left\{ \begin{array}{l}
a_*\cO_X(K_X)=(Q^{-1})^{\oplus p_g(F)}, \text{ for\; some } Q\in \PE; \\

|2K_X+P-Q|=|L_{2K}|+\cD_{3K, P}+E_2; \\

a_*\cO_X(2K_X-Q) = \cO_E^{h^0(2K_X)} \otimes \cE_{3K}^+; \\

a_*\cO_X(3K_X) = \cO_E^{h^0(3K_X)} \otimes \cE_{3K}^+;

\end{array}\right.
$$
where the first formula holds because we know that have the decomposition formula for $a_{X*}\cO_X(K_X)$ (see the main theorem in \cite{PPS} ) and the
fact that $V^0(K_X)=\{Q\}$ by Lemma \ref{bundle1} and, as in Remark \ref{modification}, we may also assume that $|L_{2K}|$ is base point free.
For simplicity of notations, we may also write $\cD_P$ instead of $\cD_{3K, P}$.

As before, we denote by $L_{3K, F}$ (resp. $L_{2K, F}$) the restriction of $L_{3K}$ (resp. $L_{2K}$) to a general fiber $F$ of $a_X$. We note that
$\cD_{P_1, F}\simeq \cD_{P_2, F}$ for $P_1$, $P_2\in \Pic^0(E)$ and we denote by $\cD_F$ this divisor.

We know that the natural map \begin{eqnarray}\label{iso1} W_3^0\otimes W_3^+\xrightarrow{\cdot E_3} H^0(F, 3K_F)\end{eqnarray} is an isomorphism, where
$W_3^0$ and $W_3^+$ defined as in the previous section. Moreover, let $W_2^0:=\cO_E^{h^0(2K_X)}\otimes k(x) \subset H^0(F, L_{2K, F})$, where $x=a(F)$. Then we also have the
natural isomorphism \begin{eqnarray}\label{iso2} W_2^0\otimes W_3^+\xrightarrow{\cdot E_2} H^0(F, 2K_F).\end{eqnarray}

We  summarize some important corollaries of Lemma \ref{bundle} and Lemma \ref{bundle1}, which will be used  in this section.

\begin{lemm}\label{free} Let $\sigma_F: F \to F_0$ be the contraction map to the minimal model of $F$. Suppose that $|mK_{F_0}|$ is base-point-free for $m=2$ or $3$, then $L_{m, F}+\cD_{ F}=m\sigma^*K_{F_0}$.  Moreover, $|L_{m, F}|$
and $|\cD_{F}|$ are base point free.

Furthermore, let $\Lambda_m={\sigma_F}_*(L_{m,F})$ and $\Delta={\sigma_F}_*( \cD_F)$. Then both are base point free and $L_{m,F}= \sigma_F^*( \Lambda_m)$, $\cD_F=\sigma_F^*( \Delta)$.
\end{lemm}

\begin{proof}
Let $E:=K_{F/F_0}$.
Then $|mK_F|=\sigma^*|mK_{F_0}|+mE$. From the isomorphism (\ref{iso1}, \ref{iso2}), we conclude that $L_{m, F}+\cD_{ F}=m\sigma^*K_{F_0}$.  Moreover, $|L_{m, F}|$ and $|\cD_{F}|$ are base point free by Lemma \ref{decomp}.1. Hence $0\leq L_{m, F}\cdot E\leq m\sigma^*K_{F_0}\cdot E=0$ and $0\leq \cD_{F}\cdot E\leq m\sigma^*K_{F_0}\cdot E=0$. Then it is clear that $L_{m,F}= \sigma_F^*( \Lambda_m)$, $\cD_F=\sigma_F^*( \Delta)$.
\end{proof}

\begin{cor} If $p_g(F) >0$, then  $P_2(X)=h^0(F_0, \Lambda_2)> 1$.
\end{cor}

\begin{proof}
Note that $P_2(X)=\textrm{rk}(\cE_{M_2}^0)$ and $h^0(F_0, \Lambda_2) = h^0(F, L_{2,F}) \ge \textrm{rk}(\cE_{M_2}^0)$.
Hence it suffices to show that $\textrm{rk}(\cE_{M_2}^0) >1$.

Suppose on the contrary that $ \textrm{rk}(\cE_{M_2}^0) =\dim W_2^0 = 1$. Then by Corollary \ref{decomp}.2, one has $h^0(F, L_{2,F})=h^0(F_0, \Lambda_2)=1$.
Moreover, $|2 K_{F_0}|= \Lambda_2 +|\Delta|$. Since  $|2K_{F_0}|$ is base point free by Proposition \ref{bps}, $\Lambda_2=0$ and hence so is $L_{2,F}$.
Now Lemma \ref{bundle1}.5 gives $L_{3,F}=L_{2,F}+K_F = K_F$. Hence $W_3^0\subset H^0(F, L_{3K, F})\simeq H^0(F, K_F)$.
Also, $|2K_{F_0}|= |\Delta|$ and  $\cD_{ F}=\sigma^*(2K_{F_0})$ and hence
 $W_3^+=H^0(F, \cD_{3K, F})\simeq H^0(F_0, 2K_{F_0})$.
It follows that the isomorphism $$\Phi_W: W_3^0 \otimes W_3^+ \cong W_3^0 \otimes H^0(F,2K_F)  \cong H^0(F, 3K_F)$$ is compatible with $\Phi: H^0(K_F) \otimes H^0(2K_F) \to H^0(3K_F)$. For any subspace $W \subset H^0(K_F)$ of dimension $\ge 2$, it is obvious that the induced map
$\Phi|_{W \otimes H^0(2K_F)}: W \otimes H^0(2K_F) \to H^0(3K_F)$ can not be injective. Therefore, $\dim W_3^0=1$ and hence $h^0(2K_F)=h^0(3K_F)$, which is the desired contradiction.
\end{proof}

The rest of this section is devoted to the proof of the following theorem.

\begin{theo}\label{nonzero} Let $X$ be a special irregular $3$-fold. Then $|5K_X|$ induces a birational map of $X$ if any one of the following conditions holds
\begin{enumerate}
\item  $p_g(F)>0$;
\item $P_2(X) \ge 2$ and $K_{F_0}^2 \ge 3$.
\end{enumerate}
\end{theo}



\begin{proof}

 Let $\varphi: X\rightarrow \mathbb P^a$ be the morphism induced by $|L_{2K}|$ and let $S$ be an irreducible component of a general pencil of $|L_{2K}|$.
 It is clear that $S$ is a smooth of general type.  We first claim that the induced map $S \to E$ is surjective.
 To see the claim, suppose on the contrary that $ S \to E$ is not surjective, then $S=F$ and hence $L_{2K}=F$ (resp. $L_{3K}$). It follows that $ V^0(X, L_{2K})=\PE$. This is a contradiction since $V^0(X, L_{2K}) \ne \PE$ by Lemma \ref{integral}.

We will also use the following notation: let $V$ be a $\mathbb Q$-divisor on $X$ such that $S\nsubseteq \mathrm{Supp}(V)$,
then we denote by $V_S$ the restriction of $V$ on $S$.


We denote by $\pi: X\rightarrow X_0$ the morphism from $X$ to its canonical model $X_0$ and let $\sigma_S: S\rightarrow S_0$ be the contraction to its
minimal model. Then $\pi^*K_{X_0}$ is a $\mathbb Q$-Cartier nef and big divisor on $X$. We then consider the refined uniform decomposition
$|2K_X-Q|=|L_{2K}|+\cD_{2K,0}+E_2$ and the short exact sequence:
\begin{multline*}
0\rightarrow \cO_X(K_X+\lceil 2K_X+2\pi^*K_{X_0}-E_2\rceil-S)\rightarrow  \cO_X(K_X+\lceil 2K_X+2\pi^*K_{X_0}-E_2\rceil)\\
\rightarrow \cO_S\big((K_{X}+\lceil 2K_X+2\pi^*K_{X_0}-E_2\rceil)_S \big)\rightarrow 0.
\end{multline*}
Note that $2K_X-S-E_2\sim \cD_{2K, Q}=:\cD_Q$ is nef and $\pi^*K_{X_0}$ is nef and big. Thus, by Kawamata-Viehweg vanishing theorem, we have the surjective map
$$H^0(X, \cO_X(K_X+\lceil 2K_X+2\pi^*K_{X_0}-E_2\rceil))\rightarrow H^0(S, K_S+ \lceil2\pi^*K_{X_0} +\cD_Q\rceil_S).$$ We also know that $5K_X\succeq
K_X+\lceil 2K_X+2\pi^*K_{X_0}-E_2\rceil$ and $5K_X-L_{2K}$ are effective divisors. Hence in order to conclude the proof of Theorem \ref{nonzero}, we just need to show that under the respective assumptions, the global section of $K_S+ \lceil2\pi^*K_{X_0} +\cD_Q\rceil_S$ induces a
birational map of $S$.

By the extension theorem (see the main theorem of \cite{K}), we know that the restriction map $H^0(X, m(K_X+L_{2K}))\rightarrow H^0(S, mK_S)$  is
surjective for any $m\geq 2$. We may take $m$ sufficient large and divisible and conclude that $\mathrm{Mov}(|m(K_X+L_{2K})|)_S\succeq
m\sigma_S^*K_{S_0}$, where $\mathrm{Mov}(|m(K_X+L_{2K})|)$ denotes the moving part of the linear system $|m(K_X+L_{2K})|$.
 Since $\cD_P$ is nef and contains no base divisor for each $P\in \Pic^0(E)$, we have $\textrm{Mov}|m\cD_Q|=|m\cD_Q|$ for $m \ge 2$ and hence
 $$\mathrm{Mov}(|m(K_X+L_{2K})|)+m\cD_Q \preceq  \mathrm{Mov} (|m(K_X+L_{2K}+\cD_Q|) \preceq   3m\pi^*K_{X_0}.$$ Thus,
$$\sigma_S^*K_{S_0}+\cD_{Q, S}\preceq   \mathrm{Mov}(|m(K_X+L_{2K})|)_S   +\cD_{Q,S} \preceq  3\pi^*K_{X_0,S}$$ and hence \begin{eqnarray}\label{equal-KK}2\pi^*K_{X_0, S}+\cD_{Q, S} \succeq
\frac{2}{3}\sigma_S^*K_{S_0}+\frac{5}{3}\cD_{Q,S}.
\end{eqnarray}

We then apply Theorem \ref{adjoint-bundle} to show  that the global section of $K_S+ \lceil2\pi^*K_{X_0} +\cD_Q\rceil_S$ induces a birational map of $S$.
We distinguish two cases depending on whether $\cD_P$ is a fiber of $a_X$ or not for $P$ general.\\

\noindent
{\bf Case 1. $\cD_P$ is a general fiber of $a_X$.}\\
In this case, $F=\cD_P$ is a general
fiber of $a_X$ and hence $\cD_{P, F}=0$. We regard $\Gamma \colon=F\cap S$ as a general fiber of $S\rightarrow E$ or a general member of the linear system $|L_{2K, F}|$ on $F$.
Note that the morphism $S\rightarrow E$ factors through its minimal model $S_0$, hence $\Gamma \cdot K_S = \Gamma \cdot \sigma_S^* K_{S_0}$.
Thus considering $\Gamma$ as a curve on $S$ yields
\begin{equation}\label{curve1}
\sigma_S^*K_{S_0}\cdot \cD_{Q,S}=\sigma_S^*K_{S_0}\cdot \Gamma =K_S\cdot \Gamma=2g(\Gamma)-2.
\end{equation}
On the other hand, considering $\Gamma$ as a curve on $F$ leads to
$$2g(\Gamma)-2= (K_F+S|_F) \cdot S|_F.$$


\noindent{\bf Claim 1.1.} $\sigma_S^*K_{S_0} \cdot \cD_{P,S} \ge 6$ and hence  $(\frac{2}{3}\sigma_S^*K_{S_0}+\frac{5}{3}\cD_{P,S})^2 \ge \frac{40}{3}$. \\
 If $|2K_{F_0}|$  is base point free (which is the case when $p_g(F)\neq 0$ or $K_{F_0}^2\geq 5$ by \cite[Theorem 5.1 and Theorem 7.4]{BPV}), then
 $L_{2K, F}=2\sigma_F^*K_{F_0}$ by Lemma \ref{free}. It follows that
\begin{equation}\label{curve2}
  \sigma_S^*K_{S_0}\cdot \cD_{Q,S}=2g(\Gamma)-2=  (K_F+\sigma_F^* 2K_{F_0} ) \cdot \sigma_F^* 2K_{F_0}=6K_{F_0}^2 \ge 6
\end{equation}

 If $|2K_{F_0}|$ has base points (which could happen when $K_{F_0}^2=3$ or $4$), we note that still have $|2K_F|=|L_{2K}|_{ F}+E_{2, F}$. By the
 assumption that $K_{F_0}^2\geq 3$ and the main result of \cite{X2}, we know that $|L_{2K}|$ induces a generically finite morphism of $F$ and $h^0(F,
 L_{2K, F})=P_2(F)\geq 4$.  It is easy to see that $(L_{2K, F})^2\geq 2(P_2(F)-2) \ge 4$. Hence
 $K_F\cdot L_{2K, F}\geq \sigma_F^*K_{F_0}\cdot L_{2K, F}\geq  4$ by Hodge index theorem. 
 Thus
 $\sigma_S^*K_{S_0}\cdot \cD_{P,S}  =(K_F+S|_F) \cdot S|_F \geq 8$.

\noindent{\bf Claim 1.2.} $(\frac{2}{3}\sigma_S^*K_{S_0}+\frac{5}{3}\cD_{P,S}) \cdot C \ge 3$ for a  general curve $C$ on $S$. \\
Suppose first that $a(C)$ is a point, then $C=\Gamma$.
Therefore $\sigma_0^* K_{S_0}\cdot C\geq 6$ and  $(\frac{2}{3}\sigma_S^*K_{S_0}+\frac{5}{3}\cD_{P,S}) \cdot C \ge 4$.

If $a(C)$ is not a point, then $\cD_{P,S}\cdot C\geq 1$. The claim holds if $\sigma_S^* K_{S_0} \cdot C \ge 2$. If $\sigma_S^* K_{S_0} \cdot C =0$, one has a contradiction by Hodge Index Theorem. If $\sigma_S^* K_{S_0} \cdot C =1$, then Hodge Index Theorem implies that $C \equiv \sigma_S^* K_{S_0}$. Then Claim 1 gives  $\cD_{P,S}\cdot C\geq 6$ and hence $(\frac{2}{3}\sigma_S^*K_{S_0}+\frac{5}{3}\cD_{P,S}) \cdot C \ge \frac{32}{3}$.

Finally, we conclude by Theorem \ref{adjoint-bundle} that  the global section of $K_S+
\lceil2\pi^*K_{X_0} +\cD_P\rceil_{S}$ induces a birational map of $S$.\\


\noindent {\bf Case 2. $a_{X*}\cD_P$ has rank $>1$}. \\
We start with the following observation.\\
\noindent
{\bf Claim 2.1.} $\cD_{P,S}$ is nef and big. Moreover, $\cD_{P,S}^2 \ge 2$.
  \begin{proof}
 We already know that $\cD_P$ is nef. It suffices to see that $\cD_{P,S}$ is big. First note that $\cD_P\equiv \cD_Q$ for $P, Q\in\Pic^0(X)$ and $S\rightarrow E$ is a surjective morphism. Hence the Iitaka model of $\cD_P$ dominates $E$ (and so does $\cD_{P,S}$).

   Suppose that $\cD_{P,S}$ is not big, then it is the pull-back on $S$ of an effective divisor of $E$. We may write $\cD_{P,S}=a^*N$ for some ample divisor $N$. Let $C$ be connected component of $S \cap F$. One has that $g(C) \ge 2$ since $S$ is of general type and the restriction $\cD_{P,C}=0$.

We then consider the short exact sequence  $$0\rightarrow \mathcal O_X(\cD_P-S)\rightarrow \mathcal O_X(\cD_P)\rightarrow \mathcal O_S(\cD_{P,S})\rightarrow 0.$$
It follows from $\textrm{rk} (a_*\cD_P) >1$ that $a_{X*}\mathcal O_X(\cD_P-S) \ne 0$ and hence $\cD_{P,F} \ge S|_F$. Further restriction to $C$ yields $0 =\cD_{P,C} \ge S|_C$. However, since $2K_X=\cD_{P}+S+E_2$, $2K_X+2S = \cD_P+3S+E-2$. Its restriction to $S$ and further to  $C$ leads to
$$0<(2K_S)|_C= \cD_{P,C}+3S|_C + E_2|_C = 0,$$ which is a contradiction.

  We now prove that $\cD_{P,S}^2 \ge 2$. If $S$ is of maximal Albanese dimension, we know by \cite[v5, Theorem 6.7]{BPS} that $\vol(S, \cD_{P,S})=\cD_{P,S}^2\geq\frac{2r}{2r-1}h^0(S, \cD_{P,S}+Q)$, where $Q\in\Pic^0(S)$
  general and $r$ is the least positive integer such that $rK_S\succeq V$. Hence $\cD_{P,S}^2\geq 2$.

  If $a_S: S\rightarrow  A_S$ maps onto a curve. Then we may apply Proposition \ref{J} to conclude that $\cD_{P,S}^2\geq 2$.
  \end{proof}

\noindent{\bf Claim 2.2.} $(\frac{2}{3}\sigma_S^*K_{S_0}+\frac{5}{3}\cD_{P,S})^2 >10$.\\
 By Claim 1,  $(\cD_{P,S})^2\geq 2$. Hence by Hodge index theorem, $\sigma_S^*K_{S_0}\cdot \cD_{P,S}\geq 2$ and hence
$$(\frac{2}{3}\sigma_S^*K_{S_0}+\frac{5}{3}\cD_{P,S})^2\geq \frac{94}{9}>10.$$

\noindent{\bf Claim 2.3.}  $(\frac{2}{3}\sigma_S^*K_{S_0}+\frac{5}{3}\cD_{P,S}) \cdot C \ge 3$ for a  general curve $C$ on $S$. \\
Let $C$ be a general curve on $S$.
We know that $\cD_{P,S}\cdot C\geq 1$ since $\cD_{P,S}$ is big and $C$ is general. If $\sigma_S^* K_{S_0} \cdot C \ge 2$ then the Claim follows.
If $\sigma_S^* K_{S_0} \cdot C =0$, one has a contradiction by Hodge Index Theorem. If $\sigma_S^* K_{S_0} \cdot C =1$, then Hodge Index Theorem implies that $C \equiv \sigma_S^* K_{S_0}$. Hence
$$(\frac{2}{3}\sigma_S^*K_{S_0}+\frac{5}{3}\cD_{P,S}) \cdot C =(\frac{2}{3}\sigma_S^*K_{S_0}+\frac{5}{3}\cD_{P,S}) \cdot \sigma_S^* K_{S_0} \ge 4.$$

 With these Claims, we thus then conclude that  $K_S+\lceil 2\pi^*K_{X_0}+\cD_{P}\rceil_{S}$ induce a
birational map of $S$.
\end{proof}

 \section{Irregular threefolds with $p_g(F)=0$}

 We always assume that $X$ is special. Then by Lemma \ref{integral}, for $P\in \Pic^0(X)$ general, we have $$|3K_X+P|=|L_{3K}|+\cD_{3K, P}+E_3,$$ and
 $$|2K_X+P|=|L_{2K}|+\cD_{2K, P}+E_2.$$
 For simplicity of notations, we write $\cD_P:=\cD_{3K, P}$ and $T_P:=\cD_{2K, P}$.
by Lemma  \ref{modification},  we may and will assume that $|L_{3K}|$ and $|L_{2K}|$ are base point free, $\cD_P$ and $T_P$ are smooth nef
divisors for $P$ general.

\noindent {\bf Warning.}
 Note that a difference from last section is that there is no reason here  that $P_2(X)>1$ and hence $L_{2K}$ could be $\cO_X$. Moreover,  there is also
 no reason that $\cD_P$ and $T_P$ are  algebraically equivalent.

 Since $p_g(F)=0$, it follows that $q(F)=0$ and $\chi(\cO_F)=1$.
  Hence  $1\leq K_{F_0}^2\leq 9$ by Bogomolov-Miyaoka-Yau inequality.
It is then leads to   the following table
 \begin{center}
    \begin{tabular}{| l | l | l | l | l | l | l | l | l | l |}
    \hline
    $K_{F_0}^2$ & 1&2&3&4&5&6&7&8&9 \\ \hline
    $P_2(F_0)=1+K_{F_0}^2$& 2&3&4&5&6&7&8&9&10\\
     $P_3(F_0)=1+3K_{F_0}^2$&4&7&10&13&16&19&22&25&28\\
      \hline
    \end{tabular}
\end{center}

We know by Theorem \ref{surface} that $|3K_{F_0}|$ induces a birational map of $F_0$ for all surfaces $F_0$ of general type with geometric genus $0$.

Since $p_g(F)=q(F)=0$, it follows that $a_* \omega_{X}=R^1a_*\omega_{X} =0$. Moreover, $R^2a_*\omega_{X}= \omega_E$ by \cite[Prop. 7.6]{Ko}. It follows that $\chi(\mathcal{O}_X)=-\chi(\omega_X)=0$ in this section.

\subsection{$\cD_{P}$ is a general fiber of $a_X$}
We will need following lemmas.

\begin{lemm}\label{prepare1}
If $\cD_{P}$ is a general fiber of $a_X$, then $L_{3K}$ is big and nef. Thus, let $S\in |L_{3K}|$ be a general member, we have $q(S)=1$ and $S$ dominates $E$.
\end{lemm}





\begin{proof}
Suppose that $\cD_P$ is a general fiber of $a_X$. Then one has that $\cV_{3K}^+ = \cE_{3K}^+$ is a line bundle of degree $1$.
By Corollary \ref{product}.4., we conclude that
\begin{equation}\label{system}
|L_{3K}|_{F} {+E_{3,F}} {\cong}  |3K_X|_F \simeq |3K_F| \cong |\sigma_F^* 3 K_{F_0}|+3K_{F/F_0}.
\end{equation}

 As $|3K_F|$ always induces a birational map of $F$, we see that $|L_{3K}|_F$ induces a birational morphism of $F$ and the Iitaka dimension of $L_{3K}$ is at least $2$.

If $L_{3K}$ is not big, then $|L_{3K}|$ induces a morphism $\phi\colon X\to \Sigma$ such that the restriction $\phi|_F\colon F\to \Sigma$ is birational for every general fiber $F$ of $a_X$.
This implies that $a_X\times \phi\colon X\to E\times \Sigma$ is birational, which is absurd. Therefore, $L_{3K}$ is nef and big.

The last assertion holds because $H^1(X, \mathcal{O}_X(-S))=H^2(X, \mathcal{O}_X(-S))=0$ due to Kawamata-Viehweg vanishing.
Since $S$ is big, then it is clear that $S$ dominates $E$.
\end{proof}

Under the assumption of the previous lemma, let $\Gamma$ be a general fiber of the morphism $S\rightarrow E$. Then as $|L_{3K}|_F\simeq |3K_F|$, $\Gamma$ is a general member of $\sigma_F^*|3K_{F_0}|$. Hence by Lemma \ref{non-hyperelliptic} and the remark after it, $\Gamma$ is not hyperelliptic.
\begin{lemm}\label{c2=0}
Let $S \in |L_{3K}|$ be a general member and $C$ be a very general curve on $S$ dominating $E$. Then  $C \cdot \Gamma = \deg( a|_C: C \to E) \ge 3$.
\end{lemm}
\begin{proof}[Proof of Lemma \ref{c2=0}]
As $C$ is very general, there exist an algebraic deformation $\{C_t\}$ of $C$ such that $\cup_t C_t$ covers an open subset of $S$.
Since $S$ dominate $E$ and $q(S)=1$ by Lemma \ref{prepare1}, the morphism $\Pic^0(E)\rightarrow \Pic^0(S)$ is an isogeny and it follows that $\Pic^0(S)\rightarrow \Pic^0(\Gamma)$ is trivial.
Therefore, for general $t_1, t_2$, $\cO_S(C_{t_1}-C_{t_2}) \in \Pic^0(S)$ induces $\cO_\Gamma (C_{t_1}- C_{t_2})\cong \cO_\Gamma$. This implies in particular that $C_{t_1, \Gamma} \sim C_{t_2, \Gamma}$, and moreover there exist a rational map $\Psi: \Gamma \dashrightarrow \mathbb{P}^1$ with $C_{t_i, \Gamma}$ being fibers. Since $\Gamma$ is non-hyperelliptic by Lemma \ref{non-hyperelliptic}, one sees that $C \cdot \Gamma= \deg(\Psi|_\Gamma: \Gamma \to \mathbb{P}^1) \ge 3$.
\end{proof}

\begin{prop}\label{>1}
Assume that $p_g(F)=0$, $K_{F_0}^2 \ge 2$ and $\cD_P$ is a general fiber of $a_X$ for $P$ general, then $|5K_X|$ induces a birational map.
 \end{prop}
 \begin{proof}
 The proof is similar to the first case of Theorem \ref{nonzero}. We assume that $|L_{3K}|$ is base point free and pick $S$ a general member of
 $|L_{3K}|$ and consider the short exact sequence
 \begin{multline*}
0\rightarrow \cO_X(K_X+\lceil 3K_X+\pi^*K_{X_0}-E_3\rceil-S)\rightarrow  \cO_X(K_X+\lceil 3K_X+\pi^*K_{X_0}-E_3\rceil)\\
\rightarrow \cO_S\big((K_X+\lceil 3K_X+\pi^*K_{X_0}-E_3\rceil)_{S} \big)\rightarrow 0.
\end{multline*}
As before, we need to show that $|K_S+\lceil \pi^*K_{X_0}+\cD_{P}\rceil|_{S}$ defines a birational map of $S$. Moreover, we again apply the extension
theorem \cite{K} to conclude that $$\pi^*K_{X_0,S} \succeq  \frac{1}{4}(\sigma_S^*K_{S_0}+\cD_{P,S})$$ and hence $$\pi^*K_{X_0, S}+\cD_{P,S}\succeq
\frac{1}{4}(\sigma_S^*K_{S_0}+5\cD_{P,S})= \colon \cM.$$


If $K_{F_0}^2\geq 2$, we know that $|3K_{F_0}|$ is base point free (\cite[Theorem 5.1]{BPV}). Then $L_{3K, F}=3\sigma_F^*K_{F_0}$. As $\cD_P$ is a general
fiber of $a_X$, denoted $F$,  and $\Gamma= \cD_P\cap S$.
By the similar computation as in Equation \ref{curve1}, \ref{curve2}, we have
$$ \left\{ \begin{array}{l}
\sigma^*K_{S_0}\cdot \cD_{P,S}=2g(\Gamma)-2=12K_{F_0}^2 \ge 24; \\
\cM^2\geq \frac{1}{16}(K_{S_0}^2+10\sigma^*K_{S_0}\cdot\cD_P)>15
\end{array} \right.
$$

Next we verify that   $\cM \cdot C \ge 3$ for any general curve $C$.
Let $C$ be a general  curve on $S$. If $a(C)$ is a point, then $C=\cD_{P,S}$ and hence $$\cM \cdot C \ge \frac{1}{4}\sigma^*K_{S_0}\cdot \cD_{P,S} \ge 6.$$
If $a(C)$ is not a point, then the induced map $a|_C: C \to E$ is of degree $\cD_{P,S}\cdot C\geq 2$ because $g(C)\geq 2$.
We next verify that $\sigma_S^*K_{S_0}\cdot C \ge 2$ to achieve $\cM \cdot C\geq 3$.

If $C^2=0$, then $\sigma_S^*K_{S_0}\cdot C = 2g-2 \geq 2$.
If $C^2=1$, then $\sigma_S^*K_{S_0}\cdot C \ge 2$ unless $\sigma_S^*K_{S_0} \equiv C$ by Hodge Index Theorem.
In this situation, we have $\cM \cdot C = \cM \cdot \sigma_S^* K_{S_0} >30$.
If $C^2 \ge 2$, then by Hodge Index Theorem, $\sigma_S^*K_{S_0}\cdot C\geq 2$.

Therefore,$\cM \cdot C\geq 3$.
Hence we conclude  by Theorem
\ref{adjoint-bundle} that the global sections of $K_X+\lceil 3K_X+\pi^*K_{X_0}-E_3\rceil$ separates two general points which are in different fibers of
$a_X$.

On the other hand, we know that in this case $|3K_X+P|=|L_{3K}|+\cD_P+E_3$, where $|L_{3K}|_{F}\simeq |3K_F|$. Hence
$|5K_X|_{F} \supset |3K_F|+V$ for some fixed divisor $V$ and $|3K_F|$ induces a birational map of $F$. Hence $|5K_X|$ induces a birational map of $X$,
when $K_{F_0}^2\geq 2$.
\end{proof}

\begin{cor}
Assume that $K_{F_0}^2=2, 4, 6$, then $\cD_{P}$ is a general fiber of $a_X$. Hence $|5K_X|$ induces a birational map of $X$.
\end{cor}
\begin{proof}
When $K_{F_0}^2=2, 4, 6$, $P_3(F)$ is a prime number and $P_3(X)\geq 2$ by Lemma \ref{plurigenera} and \ref{bundle}, hence $\rank \cE_{3K}^+=1$ and we know
that $\cD_{P}$ is a general fiber of $a_X$ by Corollary \ref{product}. Therefore, $|5K_X|$ is birational by Proposition \ref{>1}.
\end{proof}



\begin{prop}\label{=1}
Assume that $K_{F_0}^2=1$, then $\cD_{P}$ is a general fiber of $a_X$. Moreover, $|5K_X|$ induces a birational map of $X$.
\end{prop}

\begin{proof} This case is more involved.\\

\noindent
{\bf Step 1.} $P_3(X)=4$ and  $\cD_{P}$ is a general fiber of $a_X$.\\
If $K_{F_0}^2=1$,  then $P_3(F)=4$ and hence $P_3(X)=2$ or $4$ since $P_3(X)>1$.
Suppose on the contrary that $P_3(X)=2$. Then $H^0(F, 3K_F) \cong W_3^0 \otimes W_3^+$,
where both $W_3^0$ and $W_3^+$ are $2$-dimensional (cf. Corollary \ref{product}).
 Then the image of $|3K_F|$ is a smooth quadric in $\mathbb P^3$, which contradicts the fact that $|3K_F|$ induces a birational map of $F$.

Therefore $P_3(X)=4$. It follows that $\rank \cE_{3K}^+=1$ by Lemma \ref{bundle} and hence
$\cD_{P}$ is a general fiber of $a_X$ by Lemma \ref{product}.\\

\noindent
{\bf Step 2.}\label{prepare2}
 $L_{3K}^3\geq 2$, and $\cD_{P} \cdot L_{3K}^2\geq 7$.  Moreover, a general section of $|L_{3K, F}|$ is a smooth curve of
genus $7$ and hence  $\sigma_S^*K_{S_0}\cdot \cD_{P,S}=K_S\cdot\cD_{P,S}=12$.

The linear system $|L_{3K}|$ defines a generically finite morphism from $X$ to $\mathbb P^3$. Hence $L_{3K}^3\geq 2$.

We know that $|3K_{F_0}|$ has no fixed component and has less than $2$ simple base points (see for instance \cite[Corollary 12]{Cat}). Hence $(L_{3K,
F})^2\geq 9-2=7$.
A general member of $|3K_{F_0}|$ is already smooth with genus $7$. Hence a general member $\Gamma \in |L_{3K}|_F$, which is a proper transform of a general member in $|3K_{F_0}|$ has genus $7$.
Notice that $a|_S : S\to E$ factor through $S_0$. Therefore, the exceptional set of $\sigma_S: S \to S_0$ does not intersect $F$  and hence $\Gamma$.
If follows that $$ 12= 2g(\Gamma)-2= K_S \cdot \Gamma = \sigma_S^* K_{S_0} \cdot \Gamma=\sigma_S^* K_{S_0} \cdot \cD_{P,S} .$$

\noindent {\bf Step 3.} $K_{S_0}^2 \ge 10$. \\

Note that $(K_X+L_{3K})_S=K_S$. Thus $\sigma^*K_{S_0}\succeq (\pi^*K_{X_0})_S+ L_{3K, S}$. Hence
$$K_{S_0}^2\geq L_{3K}^3+2\pi^*K_{X_0}\cdot L_{3K}^2+(\pi^*K_{X_0})^2\cdot L_{3K}.$$ Moreover since $\pi^*K_{X_0}\succeq \frac{1}{3}(\cD_P+L_{3K})$, we
have by Lemma \ref{prepare2} that $$\pi^*K_{X_0}\cdot L_{3K}^2\geq \frac{1}{3}L_{3K}^3+\frac{1}{3}\cD_P\cdot L_{3K}^2\geq 3$$ and
\begin{eqnarray*}(\pi^*K_{X_0})^2\cdot L_{3K} &\geq &\frac{1}{3}\cD_P\cdot \pi^*K_{X_0}\cdot L_{3K}+\frac{1}{3}\pi^*K_{X_0}\cdot L_{3K}^2\\   &\geq
&\frac{1}{3}\cD_P\cdot \pi^*K_{X_0}\cdot L_{3K}+1.\end{eqnarray*}

Note that $a_X$ factors through $\pi: X\rightarrow X_0$. Moreover, for any $m$ sufficiently large and divisible and $P\in \Pic^0(E)_{\mathrm{tor}}$
torsion, $|mK_X+P|= \pi^*|mK_{X_0}+P|+mK_{X/X_0}$. Moreover, for $m\geq 2$, $a_{X*}\omega_X^m$ is IT$^0$ sheaves. Hence  $a_{X*}\omega_X^m$ is
continuously globally generated by \cite{pp1}. Thus the evaluation map $$\bigoplus_{P\in \Pic^0(E)_{\mathrm{tor}}}H^0(mK_X+P)\rightarrow H^0(F, mK_F)$$ is
surjective. We then conclude that $\pi^*K_{X_0, F}=\sigma^*K_{F_0}$.
 Thus $$\pi^* K_{X_0} \cdot \Gamma= \cD_P\cdot \pi^*K_{X_0}\cdot L_{3K}= \pi^*K_{K_0, F} \cdot_F L_{3K,F}={3K_{F_0}^2}=3.$$ Combining all the inequalities above, we have $K_{S_0}^2\geq
10$.\\ 

\noindent {\bf Step 4.}  $|K_S+\lceil \pi^*K_{X_0}+\cD_{P}\rceil|_{S}$ defines a birational map of $S$.\\

We consider $\cM=\frac{1}{4} \sigma_S^* K_{S_0} + \frac{5}{4} \cD_{P,S}$ and $\pi^*K_{X_0, S}+\cD_{P, S} \ge \cM$ as before. Now we have
 $$ (\pi^*K_{X_0, S}+\cD_{P, S})^2\geq \cM^2= \frac{1}{16}(K_{S_0}^2+10\sigma_S^*K_{S_0}\cdot\cD_{P, S}) \geq\frac{65}{8}>8.$$

Since $\pi^*K_{X_0} \cdot \Gamma \ge 3$, it follows that $|K_S+\lceil \pi^*K_{X_0}+\cD_{P}\rceil|_{\Gamma}$ is birational. Thanks to the vanishing of $H^1(S, K_S + \lceil \pi^*K_{X_0} \rceil)$, it follows that the system $|K_S+\lceil \pi^*K_{X_0}+\cD_{P}\rceil|_{S}$ separate points on $\Gamma$, that is, general fiber of $S\to E$. It remains to consider points on different fiber of $S \to E$.

We then consider $\cM \cdot C$ for a very general curve $C$  on $S$ dominating $E$.


Suppose that $C^2 \ge 2$, then $\sigma_S^*K_{S_0}\cdot C\geq 5$ by Hodge Index Theorem. Since $g(C) \ge 2$, then it is clear that $\cD_{P,S} \cdot C \ge 2$ for $a|_C: C \to E$ has degree $\ge 2$.
Hence $\cM \cdot C \ge \frac{15}{4}$.

Suppose that $C^2=1$, then $\sigma_S^*K_{S_0}\cdot C\geq 4$ by Hodge Index Theorem. Suppose furthermore that $K_{S_0} \cdot \sigma_S(C)=\sigma_S^*K_{S_0}\cdot C= 4$. Since $\sigma_S(C)^2$ is even and $\sigma_S(C)^2 \ge C^2=1$, one sees that $\sigma_S(C)^2 \ge 2$ and leads to a contradiction to Hodge Index Theorem. Therefore $ \sigma_S^*K_{S_0}\cdot C\geq 5$ and $\cM \cdot C \ge \frac{15}{4}$ in this situation.
Suppose that $C^2=0$, by Lemma \ref{c2=0}, one has $\cM \cdot C \ge \frac{15}{4}$ clearly.
Therefore, for very general curve, we always has $$ \cM \cdot C \ge \frac{15}{4}  > \frac{4}{1+ \sqrt{1/65} }.$$
We can conclude  that $|5K_X|$ induces a birational map of $X$ by Theorem \ref{adjoint-bundle} and the last paragraph of the proof of Proposition \ref{>1}.
\end{proof}


\begin{prop}\label{=3}
Assume that $K_{F_0}^2=3$, then $\cD_{P}$ is a general fiber of $a_X$. Moreover, $|5K_X|$ induces a birational map of $X$.
\end{prop}

\begin{proof} We proceed with the following setting:\\

\noindent {\bf Setting.}
Since $K_{F_0}^2 \ge 2$, we know that$|3K_{F_0}|$ is base point free (\cite[Theorem 5.1]{BPV}). We recall some notation as in last sections. Let $\Lambda_3:=\sigma_*L_{3K, F}$ and $\Delta:=\sigma_*\cD_{P,F}$.
Then $H^0(F_0, 3K_{F_0})\simeq W_3^0\otimes W_3^+$, where $W_3^0 $ (resp.
$W_3^+$) can be regarded as a subspace of $H^0(F_0, \Lambda_3)$ (resp. $H^0(F_0, \Delta)$).
Notice  that both $|W_3^0|$ and $|W_3^+|$ are free (and so are  $|\Lambda_3|$ and $|\Delta|$) by Corollary \ref{product}.3.\\

\noindent {\bf Claim 1.} None of the system $|\Lambda_3|$ or $|\Delta|$ defines a pencil. \\\\
Suppose first that  $|W_3^+|$ defines a pencil of curve, denoted $G$,  of genus $g(G) \geq 2$.  Hence $K_{F_0}\cdot \Delta=2g(G)-2$ is even.

Suppose furthermore that $|W_3^0|$ defines a pencil of curves as well, then $K_{F_0} \cdot \Lambda_3$ is even again.
However,  $$9=3K_{F_0}^2=K_{F_0}\cdot \Lambda_3+K_{F_0}\cdot \Delta.$$

Therefore, $|W_3^0|$ can not be a pencil. Now consider the situation that $|W_3^0|$ defines a
generically finite morphism from $h: F_0 \to  \mathbb P^m$, where $m=h^0(W_3^0)-1=\frac{1}{2}P_3(F)-1$ and denote by $V$ its image. Clearly, $\deg(V) \ge h^0(W_3^0)-2$.

Since $\Lambda_3^2 \ge \deg(V) \ge h^0(W_3^0)-2 = \frac{1}{2}P_3(F)-2=3$.
Suppose that  $\deg(V)=3$ then it is a cubic surface in $\mathbb{P}^4$, which is not of general type and hence
$\Lambda_3^2 \ge \deg h \cdot \deg V \ge 6$.
 In any event, we have  that $\Lambda_3^2 \ge 4$.
 Together with $$27= \Lambda_3^2+2\Lambda_3\cdot \Delta+\Delta^2=\Lambda_3^2+2(\Lambda_3+\Delta )\cdot \Delta =  \Lambda_3^3+ 6K_{F_0} \cdot \Delta=\Lambda_3^2+12(g(G)-1),$$
one sees that $g(G)=2$ and $\Lambda_3^2=15$.
However, by \cite[Th\'eor\`eme 2.2]{X}, a  surface $F$ of general type with $p_g=0$ and a pencil of genus $2$ curve should have $K_{F_0}^2\leq 2$. Hence we get a contradiction.
The above argument shows that none of the system $|\Lambda_3|$ or $|\Delta|$ defines a pencil.\\

\noindent {\bf Claim 2.}
 Then $\cD_P$ is a general fiber of $a_X$  and  $|5K_X|$ is birational. \\\\
Suppose that $K_{F_0}^2=3$, then $P_3(X)|P_3(F_0)=10$. If $P_3(X)=2$ or $5$, then either $|\Lambda_3|$  or $|\Delta|$ is a pencil by Corollary \ref{product}.1 , which is a contradiction to Claim 1. Since $P_3(X) >1$, then we have $P_3(X)=10$ and therefore $a_* \cD_P$ is of rank $1$. 
Therefore, by Corollary \ref{product}.4,  $\cD_P$ is a general fiber and hence $|5K_X|$ is birational by Proposition \ref{>1}.
\end{proof}

 \subsection{General fiber has $K_{F_0}^2= 5, 7,  8, 9$.\\}

 Note that $X$ is  special and by Theorem \ref{nonzero}, we will assume furthermore that $P_2(X)=1$. Hence  we have $$|2K_X+P|=T_P+E_2,$$ where $T_P$ is smooth and
 nef.

  Let $T=T_P$ for some general $P$. Considering the short exact sequence
\begin{multline*}0\rightarrow K_X+\lceil 2\pi^*K_{X_0}+2K_X+P-E_2-T \rceil  \rightarrow K_X+\lceil 2\pi^*K_{X_0}+2K_X+P-E_2 \rceil \\ \rightarrow
K_{T}+\lceil 2\pi^*K_{X_0}\rceil_T\rightarrow 0.
\end{multline*}
  By Kawamata-Viehweg vanishing theorem, we have the surjectivity of $$H^0(X, K_X+\lceil 2\pi^*K_{X_0}+2K_X+P-E_2 \rceil )\rightarrow H^0(T,  K_{T}+\lceil
  2\pi^*K_{X_0}\rceil_{T}).$$

  \begin{lemm}
If $K_{F_0}^2\geq 5$, then the volume $\vol(T)=T^3 \geq 10$.
  \end{lemm}

  \begin{proof}
  Since $K_{F_0}^2\geq 5$, the linear system $|2K_{F_0}|$ is base point free. Hence $T_{P, F}\sim 2\sigma^*K_{F_0}$. It is easy to see that $a_{X*}\mathcal
O_X(T)=a_{X*}(\omega_X^2\otimes P)$
  is an ample vector bundle $\mathcal V$ of degree $1$ and of rank $1+K_{F_{0}}^2$. By Proposition \ref{J}, we know that $$\vol(T)=T^3\geq
  3\frac{T_F^2}{1+K_{F_0}^2}=\frac{12K_{F_0}^2}{1+K_{F_0}^2}\geq 10,$$ when $K_{F_0}^2\geq 5$.
  \end{proof}

\begin{lemm}$P_3(X)=\dim W_3^0\geq 4$.
\end{lemm}
\begin{proof}
We have seen that $T^3\geq 10$. Moreover, $2\pi^*K_{X_0}\succeq T$.
Hence $\vol(X)=(\pi^*K_{X_0})^3\geq \frac{1}{8}T^3\geq \frac{5}{4}$. Since $\chi(\mathcal{O}_X)=0$, the Riemann-Roch formula implies that $P_3(X) \ge \frac{5}{2} \vol(X) \ge \frac{25}{8}$ and hence $P_3(X) \ge 4$.
\end{proof}

\begin{lemm} Suppose that $\chi(\mathcal{O}_X)=0$ and $P_2(X)=1$. Then $P_3(X) \le 5$.
Moreover, $P_3(X)=5$ if and only if its minimal model $X_0$ has at worst Gorenstein singularities.
\end{lemm}

\begin{proof} Let $X_0$ be the minimal model of $X$. By \cite[3.5]{CC1}, one has $$\sigma(X_0)=10 \chi (\mathcal{O}_{X_0})+5 \chi_2(X_0)-\chi_3(X_0) = 5-P_3(X) \ge 0.$$
Moveover, $\sigma(X_0)=0$ if and only if $X_0$ has at worst Gorenstein singularities.
\end{proof}

By [Theorem 1.1] of \cite{CCZ}, $|5K_{X_0}|$ is birational if $X_0$ is minimal with at worst Gorenstein singularities.
Moreover, $P_3(X) | P_3(F)$ for special irregular threefolds (cf. Lemma \ref{bundle}). Therefore, we conclude that
\begin{cor}
\begin{enumerate}
\item $|5K_X| \text{ is birational unless } \chi(\mathcal{O}_X)=0, P_2(X)=1  \text{ and } P_3(X) =4.$

\item In the case that $K_{F_0}^2=7$ or $8$, one has that  $|5K_X|$  is birational.
\end{enumerate}
\end{cor}

It remains to consider the following two cases:\\
{\bf Case 1.} $K_{F_0}^2=5$, $P_3(X)=\dim W_3^0=4$, and $\dim W_3^+ =4$. \\
{\bf Case 2.} $K_{F_0}^2=9$, $P_3(X)=\dim W_3^0=4$, and $\dim W_3^+ =7$.





\begin{prop}
In the Case 1 that $K_{F_0}^2=5$, $|5K_X|$ induces a birational map.
\end{prop}

\begin{proof} This is a involved case that we need to work harder.\\

\noindent {\bf Step 1.}
   The  system $|5K_X|_{T}$ induces a  birational map of $T$.

  To see this, it suffices to apply Theorem \ref{adjoint-bundle} to show that $|K_{T}+\lceil 2\pi^*K_{X_0}\rceil_{T}|$ induces a birational map of $T$. We note that
  $2\pi^*K_{X_0}\succeq T_P$ for $P\in \Pic^0(E)_{\mathrm{tor}}$. Hence $\lceil 2\pi^*K_{X_0}\rceil_{T}\succeq T_{P,T}$.

 We then consider the eventual  map of $X$ associated to $\mathcal{O}_X(T)$ defined in \cite{J}, which is indeed the relative evaluation morphism of $\mathcal V$:
  $$\varphi: X\rightarrow \mathbb P_E(\mathcal V).$$ We note that $\varphi$ is generically finite, whose restriction  to  a general fiber $F$ of $a_X$ is
  exactly the bicanonical morphism of $F$. Note that $\varphi^*\mathcal O_{\mathbb P_E(\mathcal V)}(1)=\mathcal O_X(T)$. Let $C $ be a general curve on
  $T\subset X$. If $C$ is supported in a fiber $F$ of $a_X$, we see that $C\cdot T=2(C\cdot \sigma^*K_{F_0})_F\geq 4$.

  Suppose that $C\rightarrow E$ is dominant and $C^2>0$, then Hodge Index Theorem shows that $C \cdot T \ge 4$. Suppose that $C \rightarrow E$ is dominant and $C^2=0$, then  Lemma \ref{c2=0} asserts that  $\deg(C\rightarrow E)\geq 3$. Hence $(C\cdot T)_X\geq 3$ in any situation. By Theorem \ref{adjoint-bundle}, $|5K_X|_T \ge |K_T+T_{T}|$ is birational.\\

\noindent {\bf Step 2.} We have $h^1(\mathcal O_T)=h^1(\mathcal O_X)=1$ and the restriction of $a_X$ on $T$ is $a_T$, the Albanese map of $T$.

Since $T$ is a big and nef divisor on $X$, we have $h^1(\mathcal O_T)=h^1(\mathcal O_X)=1$. Moreover,
since $$ T_{P,F}+E_{2,F} = (2K_X+P)|_F = 2K_F \succeq \sigma^* 2K_{F_0},$$ and $E_{2,F}$ is $\sigma$-exceptional for general $F$.   One has
$T_{P,F}=\sigma^* 2K_{F_0}$. This implies in particular that a general fiber $C_t$ of $T\rightarrow E$ over $t \in E$ is a
connected smooth member of $|2\sigma^*K_{F_0}|$ on $F$. Hence the restriction of $a_X$ on $T$ is the Albanese morphism of $T$ and the general fiber is a curve with
$g(C_t)=1+3K_{F_0}^2$.\\

\noindent {\bf Step 3}
Also, we  have $h^1(\mathcal O_{\cD_P})=1$. In particular, the Stein factorization of  the restriction of $a_X|_{\cD_P}$ on $\cD_P$ is  the Albanese
map of $a_{\cD_P}$.

Note that $H^2(X, \mathcal O_X(-\cD_P))$ is dual to $H^1(X, K_X+\cD_P)$. Since the Iitaka model of $(X, \cD_P)$ dominates $E$, we have $$h^1(X,
a_{X*}\mathcal O_X(K_X+\cD_P))=0$$ by Koll\'ar's vanishing theorem. Moreover, since $|\Delta|$ is nef and big,  a general member of $|\Delta|$ is connected (and hence so is $|\cD_{P,F}|$). It follows that $h^1(F, K_F+\cD_{P, F})=0$. Since
$\cD_P$ is movable and nef, $R^1a_{X*}\mathcal O_X(K_X+\cD_P)=0$ by the Koll\'ar's torsion-freeness of $R^1a_{X*}\mathcal O_X(K_X+\cD_P)$. We then conclude that $1=h^1(\mathcal O_X)=h^1(\mathcal O_{\cD_P})$.\\

\noindent {\bf Step 4.} $\Delta$ and $\Lambda_3$  are big and nef. 

Suppose on the contrary that $\Delta$ defines a pencil, then $|\Delta|$ defines a morphism $f:
F_0\rightarrow \mathbb P^1$ and $\Delta=f^*\mathcal O_{\mathbb P^1}(a)$ for some $a\geq 3$. Let $H=f^*\mathcal O_{\mathbb P^1}(1)$. From $3K_{F_0}=\Lambda_3+aH$,  we
have
$$ \Lambda_3 \cdot H = 3 K_{F_0} \cdot H= 6(g(H)-1),$$ and
$$45=(3K_{F_0})^2= \Lambda_3^2+12a(g(H)-1).$$    By\cite[Th\'eor\`eme 2.2]{X}, $g(H)\geq 3$,  hence we get a contradiction. Therefore, $\Delta$ defines a map with $2$-dimensional image. In particular, $\Delta$ is big. The same argument holds for $\Lambda_3$ as well.\\


\noindent {\bf Step 5.} Any dominant maps from $\cD_P$ to $T$ is birational.

Assume  that there exists a dominant map $\varphi: \cD_P\rightarrow T$. Since $h^1(\mathcal O_T)=h^1(\mathcal O_{\cD_P})$, the Albanese fiber
of $a_{\cD_P}$, denoted $C'_t$ should also dominates the Albanese fiber $C_t$ of $a_T$.

Since $\Lambda_3$ is a big divisor on $F_0$. We claim that $\Lambda_3^2 \ge 4$. To see this, we consider $\varphi_{\Lambda_3}: F_0 \to Z \subset \mathbb{P}^3$.
The $$\Delta^2 \ge \deg(\varphi_{\Lambda_3}) \deg(Z).$$
If $\deg(Z)=2,3$, then $Z$ is not of general type and hence $\deg(\varphi_{\Lambda_3}) \ge 2$. Therefore, $\Lambda_3^2 \ge 4$.

By Hodge Index Theorem, we have  that $K_{F_0}\cdot \Lambda_3 \geq 5$. Recall that
$$15=3K_{F_0}^2=K_{F_0}\cdot \Lambda_3 +K_{F_0}\cdot \Delta.$$
Thus $K_{F_0}\cdot \Delta \leq 10$. Then $\Delta^2\leq 20$ and equality holds only when $\Delta \equiv 2K_{F_0}$ by Hodge Index Theorem again.
Hence $$2g(C'_t)-2 = K_{F_0}\cdot \Delta + \Delta^2 \le 30.$$

Suppose that $g(C'_t) <16$, then we reach a contradiction since $C'_t$ dominates $C_t$ and $g(C_t)=16$. Therefore, $g(C'_t)=16$ and $\Delta \equiv 2K_{F_0}$. Moreover, $\cD_P \to T$ is birational. \\

\noindent
{\bf Step 6.}
The linear system $|5K_X|$ induces a generically finite map $\varphi$.

After birational modifications, consider $\varphi: X\rightarrow Z\subset \mathbb P^{N}$ be the morphism induced by $|5K_X|$.
 Assume that $\dim Z\leq 2$. Note that $|5K_X|_T$ induces a birational map from $T$ to its image. Hence $Z$ is birational to $T$ and in particular $q(Z)>0$. Thus $q(Z)=1$.
Thus restrict $\varphi$ to a general fiber $F$ of $a_X$, we see that the image of $F$ is a general fiber of the Albanese morphism of $Z$ and hence is a curve of genus $g\geq 2$, which is a contradiction to $p_g(F)=q(F)=0$.\\

\noindent
{\bf Step 7.}
The map $\varphi$ induced by  $|5K_X|$ has degree $\le 2$.

Let $x \ne y$ be two general points
of $X$. Take $P\in \Pic^0(X)$ such that $x\in T_P$.  If $y\in T_P$, then by Step 1, $|5K_X|$ separates $x$ and $y$.

If $y\notin T_P$, we note
that $T_P+\cD_{-P}+|L_3|$ is a sub-linear system of $|5K_X|$. Hence if $y\notin \cD_{-P}$, then $x$ and $y$ can be separated by $|5K_X|$.

Suppose now that $y\notin T_P$, $y\in \cD_{-P}$.
As $\varphi |_{T_P}$ is birational, $x$ is general and $\varphi(y)=\varphi(x)\in \varphi(T_P)$, we have $\dim(\varphi(\cD_{-P})\cap \varphi(T_P))=2$.
Since $\cD_{-P}$ is irreducible, it follows that $\cD_{-P}$ is an irreducible component of $\varphi^{-1}(\varphi(T_P))$ and there exists a dominant map $\cD_{-P}\dashrightarrow T_P$. By Step 5, one sees that $\deg(\varphi) \le 2$. \\

\noindent
{\bf Step 8.} The map $\varphi$ induced by  $|5K_X|$ is birational.

Suppose that $\varphi$ has degree $2$, then there exists a birational involution $\sigma$ acting on $X$ such that $\varphi$ is birational to the quotient map $X\dashrightarrow X/\sigma$. Take a birational model $\rho: \tilde{X}\rightarrow X$ such that we can lift the action $\sigma$ to a biregular involution $\tilde{\sigma}$ on $\tilde{X}$. Then $\tilde{\sigma}_*\tilde{T}_P=\tilde{\cD}_{-P}$, where $\tilde{T}_P$ and
$\tilde{D}_{-P}$ are respectively the proper transform of $T_P$ and $\cD_{-P}$. On the other hand,  we always have
$$|2K_{\tilde{X}}+P|=\tilde{T}_P+E_2',$$ and $$|3K_X-P|=|\rho^*L_{3K}|+\tilde{\cD}_{-P}+E_3',$$ where $E_2'$ and $E_3'$ are always exceptional divisors from $\tilde{X}$ to its relative minimal model over $E$. Apply $\tilde{\sigma}_*$ to the first equality, we have $2K_{\tilde{X}}+P'=\tilde{\cD}_{-P}+\tilde{\sigma}_*(E_2')$, where $P'=\tilde{\sigma}_*(P)\in \Pic^0(X)$. Hence $K_{\tilde{X}}-P-P'=\rho^*L_{3K}+E_3'-\tilde{\sigma}_*(E_2')$. Restricted to a general fiber $\tilde{F}$ of $a_{\tilde{X}}$, we see that $$K_{\tilde{F}}+\tilde{\sigma}_*(E_2')_{\tilde{F}}=(\rho^*L_{3K})_F+E_{3, \tilde{F}}'.$$ Since $\tilde{\sigma}_*(E_2')_{\tilde{F}}$ and $E_{3, \tilde{F}}'$ are effective exceptional divisors from $\tilde{F}$ to its minimal model $F_0$.
  It follows that $h^0(\tilde{F}, K_{\tilde{F}}+\tilde{\sigma}_*(E_2')_{\tilde{F}}) = h^0(F_0, K_{F_0})=0$ but the right hand side is effective, which is a contradiction.
\end{proof}

\begin{prop}
In the Case 2 that $K_{F_0}^2=9$, $|5K_X|$ induces a birational map.
\end{prop}

\begin{proof}
We basically use the same approach as in Case 1 with modification on some of the steps. First of all, one can easily sees that Steps 1-3 holds.\\

\noindent
{\bf Step 4$'$.} Both $\Lambda_3$ and $\Delta$ are ample.

When $K_{F_0}^2=9$, we know that the N\'eron-Severi group $\mathrm{NS}(F_0)\simeq \mathbb Z$. Let $H$ be an ample line bundle on $F_0$ such that its class
$[H]$ is a generator of $\mathrm{NS}(F_0)$. We know that $h^0(F_0, \Lambda_3)\geq 4$ and $h^0(F_0, \Delta )\geq 4$. Hence both
$\Lambda_3$ and $\Delta$ are ample.    \\

\noindent
{\bf Step 5$'$.}  Any dominant maps from $\cD_P$ to $T$ is birational.

Let $K_{F_0}=kH$, $\Lambda_3=b_1 H, \Delta=b_2H$. It is clear that the only possibilities for $(k,H^2)$ are $(1, 9)$, $(3,1)$.
A simple computation as in Step 5 shows that if   $\cD_P$ dominates $T$, then it is birational and $\Delta \equiv 6H\equiv 2K_{F_0}$, $\Lambda_3 \equiv 3H\equiv K_{F_0}$.

Then the Step 6-8 holds verbatim.

This concludes the proof of the Case 2, and hence the proof of our main theorem that $|5K_X|$ is birational for any irregular threefold of general type.
\end{proof}

\end{document}